\documentclass[a4paper,11pt]{amsart}
\usepackage[colorlinks,pdfdisplaydoctitle]{hyperref}
\usepackage[a4paper,centering]{geometry}
\usepackage{mathscinet}
\newtheorem{thm}{Theorem}[section]
\newtheorem{lem}[thm]{Lemma}
\newtheorem{prop}[thm]{Proposition}
\newtheorem{assumption}[thm]{Assumption}
\theoremstyle{definition}
\newtheorem{defn}[thm]{Definition}
\theoremstyle{remark}
\newtheorem{rem}[thm]{Remark}
\numberwithin{equation}{section}
 \DeclareMathOperator{\Div}{div}
 \DeclareMathOperator{\tr}{Tr}
 \DeclareMathOperator{\Span}{span}
 \newcommand{\Z}{\mathbf{Z}}
 \newcommand{\W}{\mathcal{W}}
 \newcommand{\E}{\mathbb{E}}
 \newcommand{\T}{\mathbb{T}}
 \newcommand{\e}{\epsilon}
 \newcommand{\mcl}{\mathcal}
\begin{document}
\title[Ergodicity of Navier-Stokes with mildly degenerate noise]{Ergodicity
of the 3D stochastic Navier-Stokes equations driven by mildly degenerate noise}
\author[M. Romito]{Marco Romito}
\address{Dipartimento di Matematica, Universit\`a di Firenze, Viale Morgagni 67/a, I-50134 Firenze, Italia}
\email{romito@math.unifi.it}
\urladdr{http://www.math.unifi.it/users/romito}
\author[L. Xu]{Lihu Xu}
\address{PO Box 513, EURANDOM, 5600 MB  Eindhoven. The Netherlands}
\email{xu@eurandom.tue.nl}
\thanks{The first author gratefully acknowledges the support of \emph{Hausdorff
Research Institute for Mathematics} (Bonn), through the \emph{Junior Trimester
Program} on \emph{Computational Mathematics}.
The second author thanks Dr. Martin Hairer and Prof. Sergio Albeverio for helpful
discussions, and thanks the hospitality of Dipartimento di Matematica, Universit\`a di Firenze. He is partly supported by \emph{Hausdorff Center for Mathematics} in Bonn.}
\date{July 7, 2009}
\begin{abstract}
We prove that the any Markov solution to the 3D stochastic Navier-Stokes
equations driven by a mildly degenerate noise  (i.\ e.\ all but finitely
many Fourier modes are forced) is uniquely ergodic. This follows by
proving strong Feller regularity and irreducibility.
\end{abstract}
\maketitle
\section{Introduction}

The well-posedness of three dimensional Navier-Stokes equations
is still an open problem, in both the deterministic and stochastic
cases (see \cite{Fef06} for a general introduction to the
deterministic problem and \cite{FR} for the stochastic one).
Although the existence of global weak solutions have been proven
in both cases (\cite{L}, \cite{FG}), the uniqueness is still unknown.
Inspired by the Hadamard definition of well-posedness for
Cauchy problems, it is natural to ask if there are ways to find
a good selection among the weak solutions to obtain additional
properties, such as Markovianity or continuity with respect to the
initial data.
\ \\

Da Prato and Debussche proved in \cite{DPD} that there exists
a continuous selection (i.\ e.\ the selection is strong Feller)
with unique invariant measure by studying the Kolmogorov equation
associated to the stochastic Navier-Stokes equations (SNSE).
Later Debussche and Odasso~\cite{DO} proved that this selection
is also Markovian. However, their approach essentially depends
on the non-degeneracy of the driving noise. A different and slightly
more general approach to Markov solutions,
which includes the cases of degenerate noise and even deterministic
equations, was introduced in \cite{FR}. Under the assumption of
non-degeneracy and regularity of the covariance, the authors also
proved that every Markov solution is strong Feller. Under the same
assumptions every such dynamics is uniquely ergodic and exponentially
mixing (\cite{R}).
However, both approaches rely on the non-degeneracy of the driving noise
to obtain the strong Feller property, and consequently ergodicity.
\ \\

The strong Feller property and ergodicity of SPDEs driven by
\emph{degenerate} noise have been intensively studied in recent
years (see for instance \cite{EH},\cite{HM1}, \cite{EM}, \cite{HM2},
\cite{R1}). For the two dimensional case there are several results
on ergodicity, among which the most remarkable one is by Hairer and
Mattingly \cite{HM1}. They prove that the 2D stochastic dynamics has
a unique invariant measure as long as the noise forces at least two
linearly independent Fourier modes. In this respect the three dimensional case is still open (only
partial results are known, see the aforementioned \cite{DPD},
\cite{FR}, \cite{R}, see also \cite{R1}, \cite{RZ})
and this paper tries to partly fill this gap.
More precisely, we will study the three dimensional Navier-Stokes
equations
\begin{equation}\label{e:NSE}
\begin{cases}
\dot u -\nu\Delta u+(u \cdot \nabla)u + \nabla p = \dot\eta,\\
\Div u=0,\\
u(0)=x,
\end{cases}
\end{equation}
on the torus $[0,2\pi]^3$ with periodic boundary conditions and forced
by a Gaussian noise $\dot\eta$. We assume that \emph{all except finitely
many Fourier modes} are driven by the noise, and prove that any Markov
solution to the problem is strong Feller and ergodic.
\ \\

Essentially, our approach combines the Malliavin calculus developed
in~\cite{EH} and the \emph{weak-strong uniqueness} principle
of~\cite{FR}. Comparing with well-posed problems, the dynamics
here exists only in the weak martingale sense and the standard
tools of stochastic analysis are not available. Hence, the
computations are made on an approximate cutoff dynamics (see Section
\ref{s:CutoffProblem}), which equals any dynamics up to a small time.
On the other hand, due to the degeneracy of the noise, the
Bismut-Elworthy-Li formula cannot directly be applied to prove
the strong Feller property. To fix this problem,
we divide the dynamics into high and low frequencies, applying
the formula only to the dynamics of high modes (thanks to
the essential non-degeneracy of the noise).
\ \\

Finally, we remark that, at least with the approach presented here,
general results such as the truly hypoelliptic case in \cite{HM1}
seem to be hardly achievable.
Here (as well as in \cite{FR}) the strong Feller property is essential
to propagate smoothness from small times (where trajectories are regular
with high probability) to all times. To overcome this difficulty and
understand how to study the general case, the second author (with one
of his collaborator) is proving in a work in progress (\cite{AX09})
some results similar to those in this paper, via the Kolmogorov equation
approach originally used in \cite{DPD}.
\ \\

The paper is organized as follows. Section \ref{s:results} gives
a detailed description of the problem, the assumptions on the
noise and the main results (Theorems \ref{t:strong-Feller-2}
and \ref{t:ergodic}). Section \ref{sec:cutoff dynamics} contains
the proof of strong Feller regularity, while Section
\ref{subsec:Malliavin calculus} applies Malliavin calculus
to prove the crucial Lemma \ref{lem:Malliavin-estimate}.
Section \ref{s:Irreducibility} shows the irreducibility of the
dynamics, the appendix contains additional details and the proofs
of some technical results.
\section{Description of the problem and main results}\label{s:results}

Before stating the main results of the paper, we recast
the problem in an abstract form, give the assumption
on the noise and recall a few known results.

\subsection{Settings and notations}

Let us start by writing \eqref{e:NSE} in an abstract form, using
the standard formalism for the equations (see Temam~\cite{T} for
details). Let $\T^3=[0,2\pi]^3$ be the three-dimensional torus,
let $H$ be the subspace of $L^2(\T^3;\mathbf{R}^3)$ of mean-zero
\emph{divergence-free} vector fields and let $\mathcal{P}$ be the projection
from $L^2(\T^3,\mathbf{R}^3)$ onto $H$. Denote by $A$ the Stokes
operator (that is, $A=-\mathcal{P}\Delta$ is the projection on $H$
of the Laplace operator) and by $B(u,v)=\mathcal{P}(u\cdot\nabla)v$
the projection of the nonlinearity. Following Temam \cite{T},
we consider the spaces $V_\alpha=D(A^{\alpha/2})$ and in particular
we set $V = V_1$.
\ \\

Problem \eqref{e:NSE} is recast in the following form,
\begin{equation} \label{e:NSEabs}
\begin{cases}
du + [\nu A u + B(u,u)]\,dt = Q\,dW_t,\\
u(0) = x.
\end{cases}
\end{equation}
where $Q$ is a bounded operator on $H$ satisfying suitable assumptions
(see below) and $W$ is a cylindrical Brownian motion on $H$.
In the rest of the paper we shall assume $\nu=1$, as its
exact value will play no essential role.
\ \\

Consider on $H$ the Fourier basis $(e_k)_{k\in\Z^3_*}$
defined in \eqref{e:fourier}
and, given $N\geq1$, let $\pi_N:H\to H$ be the projection onto
the subspace of $H$ generated by all modes $k$ such that
$|k|_\infty:=\max |k_i|\leq N$.
\begin{assumption}[Assumptions on $Q$]\label{a:Q}
The operator $Q:H\to H$ is linear bounded and there are $\alpha_0>\tfrac12$
and an integer $N_0\geq1$ such that
\begin{itemize}
\item[\lbrack A1\rbrack] (diagonality) $Q$ is diagonal on the Fourier basis $(e_k)_{k\in\Z^3_*}$,
\item[\lbrack A2\rbrack] (finite degeneracy)  $\pi_{N_0}Q=0$ and $\ker((Id-\pi_{N_0})Q)=\{0\}$,
\item[\lbrack A3\rbrack] (regularity) $(Id-\pi_{N_0})A^{\alpha_0+3/4}Q$ is bounded invertible (with bounded inverse) on $(Id-\pi_{N_0})H$.
\end{itemize}
\end{assumption}
Further details can be found in Subsection \ref{ss:geometry}. We only remark that
[A3] is essentially the same as in \cite{FR} (we restrict here to $\alpha_0>\tfrac12$
for simplicity), while [A2] is the main assumption. The restriction $\pi_{N_0}Q=0$
in [A2] (as well as property [A1]) has been taken to simplify the exposition.
\ \\

\subsection{Markov solutions}

Following the framework introduced in~\cite{FR} (to which we refer for further details),
we define the weak martingale solutions to problem \eqref{e:NSEabs}
(cfr. Definition 3.3, \cite{FR}).
\begin{defn}[Weak martingale solutions]
Given a probability measure $\mu$ on $H$, a solution $P$ to problem
\eqref{e:NSEabs} with initial condition $\mu$ is a probability
measure on $\Omega = C([0,\infty);D(A)')$ such that
\begin{enumerate}
\item the marginal at time $t=0$ of $P$ is equal to $\mu$,
\item $P[L^\infty_{loc}([0,\infty);H)\cap L^2_{loc}([0,\infty);V)] = 1$,
\item For every $\phi\in D(A)$, the process
$$
M_t^\phi = \langle\xi_t - \xi_0,\phi\rangle_H
  + \int_0^t\langle\xi_s, A\phi\rangle_H\,ds
  - \int_0^t\langle B(\xi_s,\phi),\xi_s\rangle_H \,ds
$$
is square integrable and $(M_t^\phi, \mathcal{B}_t,P)_{t\geq 0}$ is a
continuous martingale with quadratic variation $t|Q\phi|_H^2$,
\end{enumerate}
where $(\xi_t)_{t\geq0}$ is the canonical process on $\Omega$ and
$\mathcal{B}_t$ is the Borel $\sigma$-field of $C([0,t];D(A)')$.
\end{defn}
A Markov solution $(P_x)_{x\in H}$ to problem \eqref{e:NSEabs} is
a family of weak martingale solutions such that $P_x$ has initial
condition $\delta_x$ and the \emph{almost sure Markov property}
holds: for every $x\in H$ there is a Lebesgue null-set
$T_x\subset(0,\infty)$ such that for every $t\geq0$ and all
$s\notin T_x$,
\begin{equation}\label{e:asmarkov}
\E^{P_x}[\phi(\xi_{t+s})|\mathcal{B}_s] = \E^{P_{\xi_s}}[\phi(\xi_t)],
\qquad P_x-\text{a.\ s.}
\end{equation}
Existence of at least a Markov solution is ensured by Theorem 3.7 of
\cite{FR} (see also~\cite{Flarom06}, \cite{GolRocZha08}), for
weak martingale solutions that satisfy either a super-martingale
type energy inequality (\cite{FR}, see also \cite{GolRocZha08} where the authors
give an amended version) or an almost sure energy balance (\cite{Rom08b}).
More details on the martingale problem associated to these equations can be
found in \cite{Rom08a}. Given a Markov solution $(P_x)_{x\in H}$, define the
\emph{a.\ s.\ transition semigroup} $P_t:\mathcal{B}_b(H)\to\mathcal{B}_b(H)$
as
\begin{equation*}
P_t\phi(x) = \E^{P_x}[\phi(\xi_t)].
\end{equation*}
Thanks to \eqref{e:asmarkov}, for every $x\in H$, there is a Lebesgue null-set
$T_x\subset(0,\infty)$ such that $P_{t+s}\phi(x)=P_s P_t\phi(x)$
for all $t\geq0$ and all $s\not\in T_x$.
\subsection{A regularized cut-off problem} \label{s:CutoffProblem}

The dynamics \eqref{e:NSE} is dissipative, hence it is possible to prove
existence of a unique local solution up to a small random time. Within
this time, the solution to the following equation \eqref{e:NSE-cutoff}
coincides with any Markov solution. Let us make this rough observation
more precise.

Let $\chi:[0,\infty)\to[0,1]$ be a smooth function such that
$\chi(r)\equiv1$ for $r\leq 1$ and $\chi(r)\equiv0$ for $r\geq2$.
Set
$$
\W = V_{2\alpha_0+\frac12},\qquad
\W' = V_{-(2\alpha_0+\frac12)},\qquad
\widetilde{\W} = V_{2\alpha_0+\frac34},
$$
(where $\alpha_0$ is the constant in the Assumption~\ref{a:Q}). Given $\rho>0$,
and $x\in\W$, consider
\begin{equation} \label{e:NSE-cutoff}
\begin{cases}
du^\rho + [A u^\rho + B(u^\rho,u^\rho)\chi(\frac{|u^{\rho}|_\W}{3\rho})]\,dt = Q(u^\rho)\,dW_t\\
u^{\rho}(0) = x,
\end{cases}
\end{equation}
where
$$
Q(u) = Q + \bigl(1 - \chi(\frac{|u|_\W}{\rho})\bigr)\overline{Q}
$$
and $\overline{Q}$ is a non-degenerate operator on $\pi_{N_0}H $ (see \eqref{e:LowCov}
for a detailed definition). It is easy to see that $Q(u)$ is non-degenerate as $|u|_\W\leq\rho$.

\begin{thm}[Weak-strong uniqueness]\label{t:weakstrong}
For every $x\in\W$, there exists a unique weak solution to \eqref{e:NSE-cutoff} so that
the associated distribution $P^\rho_x$ satisfies $P^\rho_x[C([0,\infty);\W)]=1$.
Moreover, given $\rho\geq1$, define $\tau_\rho:\Omega\to[0,\infty]$ by
$$
\tau_\rho(\omega) = \inf\{t\geq 0: |\omega(t)|_\W\geq\rho\},
$$
(and $\tau_\rho(\omega)=\infty$ if the set is empty). If $x\in\W$
and $|x|_\W<\rho$, then on $[0,\tau_\rho]$, $P^\rho_x$ coincides with
any Markov solution $(P_x)_{x\in\W}$ of \eqref{e:NSEabs}, i.\ e., for all
$t>0$ and $\phi\in\mathcal{B}_b(H)$,
\begin{equation}\label{e:matchcutoff}
\E^{P^\rho_x}[\phi(\xi_t) 1_{\{\tau_\rho\geq t\}}]
 = \E^{P_x}[\phi(\xi_t) 1_{\{\tau_\rho\geq t\}}].
\end{equation}
Finally, if $|x|_\W<\rho$, then
\begin{equation}\label{e:stoppingtime}
\lim_{\epsilon\rightarrow0}P^\rho_{x+h}[\tau_\rho\geq\epsilon]=1,
\end{equation}
uniformly for $h$ in any closed subset of $\{h\in\W: |x+h|_\W<\rho\}$.
\end{thm}
\begin{proof}
Existence and uniqueness for problem~\eqref{e:NSE-cutoff} are standard,
since the nonlinearity and the operator $Q(u^\rho)$ are Lipschitz.
Let $\widetilde{u}^\rho$ be the solution to problem~\eqref{e:NSE-cutoff}
with $Q(u^\rho)$ replaced by $Q$, then $\tau_\rho(u^\rho) = \tau_\rho(\widetilde u^\rho)$.
By pathwise uniqueness, $u^\rho(t) = \widetilde u^\rho(t)$ on $[0,\tau_\rho]$.
This immediately implies \eqref{e:matchcutoff} and \eqref{e:stoppingtime} by
Theorem 5.12 of~\cite{FR}.
\end{proof}

\subsection{Main results}

The strong Feller and ergodicity results of \cite{FR}, \cite{FlaRom07}, \cite{R}
are obtained under a strong non-degeneracy assumption on the covariance.
This paper relaxes this assumption, as shown by the following results.
\begin{thm} \label{t:strong-Feller-2}
Assume Assumption~\ref{a:Q}. Let $(P_x)_{x\in H}$ be
a Markov solution to~\eqref{e:NSEabs}, and let $(P_t)_{t\geq0}$ be
the associated transition semigroup. Then $(P_t)_{t\geq0}$ is strong
Feller in $\W$.
\end{thm}
\begin{proof}
The theorem is a straightforward application of Theorem 5.4 of
\cite{FR}, once Theorems~\ref{t:weakstrong} and
\ref{thm:strong Feller 1} are taken into account.
\end{proof}

\begin{thm}\label{t:ergodic}
Under the same assumptions of the previous theorem, every Markov
solution $(P_x)_{x\in H}$ to~\eqref{e:NSEabs} is uniquely ergodic and
strongly mixing. Moreover, the (unique) invariant measure $\mu$
corresponding to a given Markov solution is \emph{fully supported}
on $\W$, i.\ e.\ $\mu(\W)=1$ and $\mu(U)>0$ for every open set $U$ of
$\W$.
\end{thm}
\begin{proof}
Given a  Markov solution $(P_x)_{x\in H}$, there exists at least one
invariant measure (Theorem~3.1, \cite{R}). Uniqueness follows from Doob's
theorem (Theorem 4.2.1 of \cite{DPZ}), since by Theorem~\ref{t:strong-Feller-2}
and Proposition \ref{p:support} the system is both strong Feller and
irreducible. The claim on the support follows again from Proposition
\ref{p:support}.
\end{proof}

\begin{rem}
The strong Feller estimate on the transition semigroup can be made
more quantitative with the same method used in \cite{FlaRom07},
but unfortunately this only gives a Lipschitz estimate for the semigroup
up to a logarithmic correction (compare with \cite{DPD}).
\ \\

Moreover, by Theorem 3.3 of \cite{R}, the convergence to the invariant
measure is exponentially fast, if the Markov solutions satisfy
an almost sure version of the energy inequality (see \cite{R},
\cite{Rom08b}). The theorem in \cite{R} is proved under an
assumption of non-degeneracy of the noise, but the only arguments
really used are that the dynamics is strong Feller and irreducible.
\end{rem}

\section{Strong Feller property of cutoff dynamics} \label{sec:cutoff dynamics}

This section will mainly prove the following theorem:
\begin{thm} \label{thm:strong Feller 1}
There is $\rho_0>0$ (depending only on $N_0$ and $Q$) such that
for $\rho\geq\rho_0$ the transition semigroup $P^{\rho}_t$ associated
to equation \eqref{e:NSE-cutoff} is strong Feller.
\end{thm}

Fix $N\geq N_0$ (whose value will be suitably chosen later in Proposition
\ref{prop:modified Hormander}). In this and the following section we
shall denote with the superscript $L$ the quantities projected onto the
modes smaller than $N$ and with the superscript $H$ those projected onto
the modes larger than $N$. We divide the equation \eqref{e:NSE-cutoff}
into the low and high frequency parts (\emph{dropping the $\rho$ in
$u^{\rho}$} for simplicity),
\begin{equation} \label{e:low and high frequency part equation}
\begin{cases}
du^L+[A u^L+B_L(u,u)\chi (\frac{|u|_\W}{3\rho})]\,dt=Q_L(u)dW^L_t\\
du^H+[A u^H+B_H(u,u)\chi (\frac{|u|_\W}{3\rho})]\,dt=Q_HdW^H_t
\end{cases}
\end{equation}
where $u^L=\pi_N u$, $u^H=(Id-\pi_N)u$, $W^L=\pi_N W$, $W^H=(Id-\pi_N)W$,
$B_L=\pi_N B$, $B_H=(Id-\pi_N)B$, $Q_L(u)=Q(u)\pi_N$ and $Q_H=Q(u)(Id-\pi_N)$.
In particular, $Q_H$ is independent of $u$.
\ \\

With the above separation for the dynamics, it is natural to define the
\emph{Frechet derivatives} for their low and high frequency parts. More
precisely, for any stochastic process $X(t,x)$ on $H$ with $X(0,x)=x$,
the Frechet derivative $D_hX(t,x)$ is defined by
\begin{equation*}
D_h X(t,x):=\lim_{\epsilon \rightarrow 0} \frac{X(t,x+\epsilon
h)-X(t, x)}{\epsilon},
\qquad h \in H,
\end{equation*}
provided the limit exists. Moreover, it is natural to define the linear
map $DX(t,x): H \rightarrow H$ by
\begin{equation*}
DX(t,x)h=D_hX(t,x), \qquad h \in H.
\end{equation*}
One can easily define $D_LX(t,x)$,
$D_H X(t,x)$, $D_LX^H(t,x)$, $D_H X^L(t,x)$ and so on in a similar way,
for instance, $D_H X^L(t,x): H^H \rightarrow H^L$ is defined by
$$
D_H X^L(t,x) h = D_{h}X^L(t,x), \qquad h\in H^H
$$
with $D_{h}X^L(t, x)=\tfrac{1}{\epsilon}\lim_{\epsilon \rightarrow 0} [X^L(t,x+\epsilon h)-X^L(t,x)]$.
\\

Let $C^k_b(\W)$ be the set of functions on $\W$ with bounded
$0$-th, $\ldots$, $k$-th order derivatives. Given a $\psi \in C^1_b(\W)$, for any $h \in \W$, the derivative of $\psi(x)$ along
$h$, denoted by $D_h \psi(x)$, is defined by
$$
D_h \psi(x)=\lim_{\epsilon\rightarrow 0} \frac{\psi(x+\e h)-\psi(x)}{\e}.
$$
Clearly, the map $D\psi(x): \W\to\mathbf{R}$, defined by $D\psi(x)h=D_h \psi(x)$
for all $h\in\W$, is linear bounded. Hence $D\psi(x)\in \W'$.
Similarly, $D_L \psi(x)$ and $D_H \psi(x)$ can be defined (e.g.
$D_L \psi(x) h=\lim_{\e \rightarrow 0}[\psi(x+\e h)-\psi(x)]/\e$,  $h\in\W^L$).
\ \\

To prove Theorem \ref{thm:strong Feller 1}, we need to approximate
\eqref{e:low and high frequency part equation}
by the following more regular dynamics:
\begin{equation} \label{e:NSE-delta}
\begin{cases}
du^{\delta, \rho}+[A u^{\delta,\rho}+\mathrm{e}^{-A_H \delta}B(u^{\delta,\rho},u^{\delta,\rho})
\chi(\frac{|u^{\delta,\rho}|_\W}{3\rho})]\,dt=Q(u^{\delta,\rho})dW_t\\
u^{\delta,\rho}(0)=x
\end{cases}
\end{equation}
where $\delta>0$ and $A_H=(Id-\pi_N)A$ (the existence and uniqueness
of weak solution to equation \eqref{e:NSE-delta} is standard).
The reason for introducing this approximation, roughly speaking,
is that one cannot prove $B(u,v) \in Ran(Q)$ but easily has
$\mathrm{e}^{-A_H \delta}B(u,v) \in Ran(Q)$, which is the key point for
finding a suitable direction for the Malliavin derivatives
(see Section \ref{subsec:Malliavin calculus}).
\\

Define \emph{two maps} $\Phi_t(\cdot)$ and $\Phi^{\delta}_t(\cdot)$ from $H$ to $H$ by
\begin{equation*}
\Phi_t(x):=u^{\rho}(t)
\qquad\text{and}\qquad
\Phi^{\delta}_t(x):=u^{\delta,\rho}(t),
\end{equation*}
where $u^{\rho}(t), u^{\delta, \rho}(t)$ are the solutions to \eqref{e:NSE-cutoff} and \eqref{e:NSE-delta} respectively.
The following proposition shows that $\Phi_t$ is the limit of
$\Phi^{\delta}_t$ as $\delta\to0^+$ in the some sense, and
will be proven in the appendix.
\begin{prop} \label{prop:approximate cutoff}
For every $T>0$ and $p \geq 2$,
there exist some $C_i=C_i(p,\rho,\alpha_0)>0$, $i=1,2$ such that
\begin{gather}
\E[\sup_{0 \leq t \leq T}|\Phi_t-\Phi^{\delta}_t|^p_\W] \leq C_1 \mathrm{e}^{C_1 T} |\mathrm{e}^{-A \delta}-Id|^p_\W,
	 \label{e:ApproxCutoff1}\\
\E[\sup_{0 \leq t \leq T}|D\Phi_t-D\Phi^{\delta}_t|^p_{\mathcal{L}(\W)}] \leq C_2 \mathrm{e}^{C_2 T} |\mathrm{e}^{-A \delta}-Id|^p_\W.
	\label{e:ApproxCutoff2}
\end{gather}
For any $\psi \in C^1_b(\W)$, $h \in \W$ and $t>0$,
\begin{equation} \label{e:limitPS}
\lim_{\delta \rightarrow 0+} |D_h\E[\psi(\Phi^{\delta}_t)]-D_h\E[\psi(\Phi_t)]|=0.
\end{equation}
\end{prop}
\ \\

The main ingredients of the proof of Theorem \ref{thm:strong Feller 1}
are the following two lemmas, i.e. Lemmas \ref{lem:Malliavin-estimate}
(proved in Section \ref{subsec:Malliavin calculus}) and
\ref{lem:low-high-freqency-estimate} (proved in the appendix, see page
\pageref{pf:low-high-freqency-estimate}).
\begin{lem} \label{lem:Malliavin-estimate}
There exists some constant $p>1$ (possibly large) such that such that for
every $x\in\widetilde{\W}$, $h \in \W^L$, $\psi \in C_b^1(H)$
and $t\geq t_0$,
\begin{equation*}
|\E[(D_L\psi)(\Phi^{\delta}_t
(x))D_h \Phi^{\delta,L}_t(x)]| \leq \frac{C \mathrm{e}^{Ct}(1+|x|_{\widetilde{\W}})^p}{t^p}
\|\psi\|_\infty |h|_\W,
\end{equation*}
where $C=C(\rho, \alpha_0)>0$.
\end{lem}
\begin{lem} \label{lem:low-high-freqency-estimate}
For
any $T>0$, $p\geq 2$ and $\delta\geq0$, there exist some $C_i=C_i(p,\alpha_0,\rho)$, $i=1,\ldots,7$, such that
\begin{gather}
\E(\sup_{0 \leq t \leq T}|\Phi^{\delta}_t(x)|^p_\W)
	\leq C_1\mathrm{e}^{C_1T}|x|_\W^{p},
	\label{e:solution-estimate1}\\
\E[\sup_{0 \leq t \leq T} |\Phi^{\delta}_t(x)|^p_{\widetilde{\W}}]
	\leq C_2\mathrm{e}^{C_2T} |x|_{\widetilde{\W}}^p,
	\label{e:W-tilde1}\\
\E[\sup_{0 \leq t \leq T} |t^{1/8}\Phi^{\delta}_t(x)|^p_{\widetilde{\W}}]
	\leq C_3\mathrm{e}^{C_3T} |x|_{\W}^p,
	\label{e:W-tilde2}\\
\E[\sup_{0 \leq t \leq T}|D_{h} \Phi^{\delta}_t (x)|_\W^p]
	\leq C_4 \mathrm{e}^{C_4 T}|h|^p_\W,
	\qquad h\in\W,
	\label{e:Frechet-derivative}\\
\E[\int_0^t |A^{1/2}D_{h} \Phi^{\delta}_s (x)|_\W^2\,ds]
	\leq C_5\mathrm{e}^{C_5t} |h|^2_\W,
	\qquad h\in\W,
	\label{e:Frechet-derivative-integral}\\
\E[\sup_{0 \leq t \leq T}|D_{h^L} \Phi^{\delta,H}_t(x)|_\W^p]
	\leq (T^{p/2}\vee T^{p/8}) C_6\mathrm{e}^{C_6 T}|h^L|_\W^p,
	\qquad h^L\in\W^L,
	\label{e:DLH}\\
\E[\sup_{0 \leq t \leq T}|D_{h^H} \Phi^{\delta,L}_t(x)|_\W^p]
	\leq  (T^{p/2} \vee T^{p/8}) C_7 \mathrm{e}^{C_7 T}|h^H|_\W^p,
	\qquad h^H\in\W^H.
	\label{e:DHL}
\end{gather}
\end{lem}
\ \\

\begin{proof}[Proof of Theorem \ref{thm:strong Feller 1}]
Here we follow the idea in the proof of Proposition 5.2 of \cite{EH}.
Set $S_t \psi(x)=\E[\psi(\Phi^{\delta}_t)]$ for any $\psi \in C_b^2(\W)$,
we prove the theorem in the following two steps.

\noindent\emph{Step 1. Estimate $DS_t \psi(x)$ for all $x \in \widetilde{\W}$}:
By Assumption \ref{a:Q}, the operator $A_H^{3/4+\alpha_0}$ is bounded
invertible on $H$, we know by \eqref{e:Frechet-derivative-integral}
that $y^H_t=Q_H^{-1}D_{h^H}\Phi_t^{\delta,H}\in H^H$ $dt\times dP-a.s.$,
hence we can proceed as in the proof of Proposition 5.2 of \cite{EH}
(more precisely, formula (5.8)) to get
\[
D_{h^H}S_t \psi(x)
 =  \frac{2}{t}\E\Bigl[\psi(\Phi^{\delta}_t)\int_{\frac{t}{4}}^{\frac{3t}{4}}\langle y^H_s,dW^H_s\rangle_H\Bigr]
  + \frac{2}{t}\int_{\frac{t}{4}}^{\frac{3t}{4}}\E[D_LS_{t-s}\psi(\Phi^{\delta}_s)D_{h^H}\Phi^{\delta,L}_s]\,ds
\]
Hence, by Burkholder-Davis-Gundy's inequality,
\begin{equation} \label{e:HHL}
\begin{split}
\big| D_{h^H}S_t \psi(x) \big|
&\leq \frac{2}{t}\|\psi\|_\infty \Bigl(\int_{\frac{t}{4}}^{\frac{3t}{4}} \E|y_s^H|^2_H\,ds\Bigr)^{\frac12}
      +\frac{2}{t}\int_{\frac{t}{4}}^{\frac{3t}{4}}\E[|D_LS_{t-s}\psi(\Phi^{\delta}_s)|_{\W'}|D_{h^H} \Phi^{\delta,L}_s|_\W]\,ds \\
&\leq \frac{C_1}{t}\mathrm{e}^{C_1 t}\|\psi\|_\infty |h^H|_\W
      +\frac{2}{t}\int_{\frac{t}{4}}^{\frac{3t}{4}}\E[|D_LS_{t-s}\psi(\Phi^{\delta}_s)|_{\W'}
 |D_{h^H} \Phi^{\delta,L}_s|_\W]\,ds
\end{split}
\end{equation}
with $C_1=C_1(p, \alpha_0, \rho)$, since by \eqref{e:Frechet-derivative-integral},
\begin{equation*}
\int_{\frac{t}{4}}^{\frac{3t}{4}} \E|y_s^H|^2_H \,ds
 =   \int_{\frac{t}{4}}^{\frac{3t}{4}} \E|Q^{-1}_{H}D_{h^H} \Phi^{\delta,H}_t|^2_H \,ds
\leq c \int_{\frac{t}{4}}^{\frac{3t}{4}} \E|A^{1/2}D_{h^H} \Phi^{\delta,H}_t|^2_\W \,ds
\leq c\mathrm{e}^{c t} |h^H|^2_\W.
\end{equation*}
For the low frequency part, according to Lemma \ref{lem:Malliavin-estimate}, there exists $C_2=C_2(\alpha_0, \rho)$ such that
\begin{equation} \label{e:LLH}
\begin{split}
|D_{h^L}S_t \psi(x)|&=|D_{h^L}S_{t/2} (S_{t/2}\psi)(x)|
=|\E[D_L S_{t/2}\psi(\Phi^{\delta}_{t/2}) D_{h^L}
\Phi^{\delta,L}_{t/2}]|+|\E[D_H S_{t/2}\psi(\Phi^{\delta}_{t/2})D_{h^L} \Phi^{\delta,H}_{t/2}]| \\
&\leq \frac{C_2 \mathrm{e}^{C_2 t}(1+|x|_{\widetilde{\W}})^p}{t^p} \|\psi\|_{\infty}
|h^L|_\W+\E[|D_H
S_{t/2}\psi(\Phi^{\delta}_{t/2})|_{\W'} |D_{h^L}
\Phi^{\delta,H}_{t/2}|_\W]
\end{split}
\end{equation}
where $p>1$ is the constant in Lemma \ref{lem:Malliavin-estimate}.
\ \\

Fix $0<T<1$, denote
$$
\psi_{T}
 = \sup_{x \in \widetilde{\W}, 0 \leq t \leq T} \frac{t^p |D S_t \psi(x)|_{\W'}}{(1+|x|_{\widetilde{\W}})^p},
$$
combine \eqref{e:HHL} and \eqref{e:LLH},
then for every $t \in (0,T]$,
\begin{equation*}
\begin{split}
|D_h S_t & \psi(x)|
\leq  \frac{C_1}{t}\mathrm{e}^{C_1 T} \|\psi\|_\infty |h|_\W+
     \frac{C_2 \mathrm{e}^{C_2 t}(1+|x|_{\widetilde{\W}})^p}{t^p} \|\psi\|_{\infty}
|h|_\W \\
&+ \psi_{T}\Bigl[\frac{2}{t}\int_{\frac{t}{4}}^{\frac{3t}{4}}\frac{1}{(t-s)^p}
\E[(1+|\Phi^{\delta}_{s}|_{\widetilde{\W}})^p|D_{h^H}\Phi^{\delta,L}_s|_\W]\,ds
                       +(\frac{2}{t})^p \E[(1+|\Phi^{\delta}_{t/2}|_{\widetilde{\W}})^p|D_{h^L}\Phi^{\delta,H}_{t/2}|_\W] \Bigr],
\end{split}
\end{equation*}
thus (noticing $0<T<1$)
\begin{equation*}
\frac{t^p |D_h S_t \psi(x)|}{(1+|x|_{\widetilde{\W}})^p}
\leq  C_3\mathrm{e}^{C_3T}\|\psi\|_{\infty}|h|_\W
       +\psi_{T} C_4 \mathrm{e}^{C_4 T} T^{1/8}|h|_\W,
\end{equation*}
where $C_i=C_i(p,\alpha_0,\rho)>0$ (i=3,4) and the previous inequality is due to
\begin{equation*}
\begin{split}
\bigl(\E[(1+|\Phi^{\delta}_{s}|_{\widetilde{\W}})^p|D_{h^H} \Phi^{\delta,L}_s|_\W]\bigr)^2
&\leq \E[\sup_{0 \leq s \leq T}(1+|\Phi^{\delta}_{s}|_{\widetilde{\W}})^{2p}]\E[\sup_{0 \leq s \leq T} |D_{h^H}\Phi^{\delta,L}_s|^2_\W] \\
&\stackrel{\eqref{e:W-tilde1},\eqref{e:DHL}}{\leq} T^{1/4} C \mathrm{e}^{C T} |h|^2_\W(1+|x|_{\widetilde{\W}})^{2p}.
\end{split}
\end{equation*}
 Hence
\begin{equation*}
\psi_{T} \leq C_3\mathrm{e}^{C_3T}\|\psi\|_{\infty}|h|_\W
       +\psi_{T} C_4 \mathrm{e}^{C_4 T} T^{1/8}|h|_\W.
\end{equation*}
From the above inequality, as $T$ is sufficiently small, we have
$$\psi_{T} \leq C_5 \|\psi\|_\infty$$
with $C_5=C_5(T,\rho,\alpha_0)>0$, thus for $0<t \leq T$,
\begin{equation}  \label{e:DSt-tilde-W}
|DS_t \psi(x)|_{\W'} \leq \frac{C_5 (1+|x|_{\widetilde{\W}})^p}{t^p} \|\psi\|_{\infty}.
\end{equation}
\ \\

\noindent \emph{Step 2. Strong Feller property of $P^{\rho}_t$.}
Applying Cauchy-Schwartz inequality, \eqref{e:DSt-tilde-W}, \eqref{e:Frechet-derivative} and
\eqref{e:W-tilde2} in order, for any $h\in\W$ and any $0<t \leq T$, we have
\begin{equation*}
\begin{split}
|D_hS_{2t}\psi(x)|^2
 &=  |\E[DS_{t}\psi(\Phi^{\delta}_{t})D_h \Phi^{\delta}_{t}]|^2
\leq \E[|DS_{t}\psi(\Phi^{\delta}_{t})|_{\W'}^2]\E[|D_h \Phi^{\delta}_{t}|^2_\W]\\
& \leq \frac{C}{t^{2p}} \|\psi\|^2_{\infty} \E[(1+|\Phi^{\delta}_{t}|_{\widetilde{\W}})^{2p}]|h|^2_\W
\leq \frac{C}{t^{9p/4}} \|\psi\|^2_{\infty}(1+|x|_\W)^{2p}|h|^2_\W
\end{split}
\end{equation*}
where $C=C(\alpha_0,\rho,T)$. Let $\delta\to 0^+$,
we have by \eqref{e:limitPS}
\begin{equation} \label{e:DhP2t}
|D_hP^{\rho}_{2t} \psi(x)| \leq \frac{C}{t^{9p/8}} \|\psi\|_{\infty} |x|_\W|h|_\W,
\qquad 0<t \leq T.
\end{equation}
Clearly, \eqref{e:DhP2t} implies that $(P^\rho_t)_{t\in(0,T]}$
is strong Feller (\cite{DPZ}). The extension of the strong Feller property to arbitrary $T>0$
is standard.
\end{proof}

\section{Malliavin Calculus and Proof of Lemma \ref{lem:Malliavin-estimate}} \label{subsec:Malliavin calculus}

In this section, we will \emph{only} study the equation \eqref{e:NSE-delta},
following the idea in \cite{EH} to prove Lemma \ref{lem:Malliavin-estimate}.
A very important point is that all the estimates in lemmas
\ref{lem:Jacobi-estimate} and \ref{lem:inverse} are \emph{independent
of $\delta$} (thanks to the cutoff and to that our Malliavin calculus
is essentially on low frequency part of $\Phi^{\delta}_t$).
We will simply write \underline{$\Phi_t=\Phi^{\delta}_t$} throughout
this section.

\subsection{Proof of Lemma \ref{lem:Malliavin-estimate}}

Given $v \in L^2_{\text{loc}}(\mathbf{R}_+, H)$, the Malliavin derivative of
$\Phi_t$ in direction $v$, denoted by $\mathcal{D}_v \Phi_t$, is
defined by
\begin{equation*}
\mathcal{D}_v \Phi_t=\lim_{\epsilon \rightarrow 0}
\frac{\Phi_t(W+\epsilon V,x)-\Phi_t(W,x)}{\epsilon}
\end{equation*}
where $V(t)=\int_0^t v(s)\,ds$. The direction $v$ can be random and is adapted
to the filtration generated by $W$.
The Malliavin derivatives on the low and high frequency parts,
denoted by $\mathcal{D}_v \Phi^L_t$ and $\mathcal{D}_v \Phi^H_t$,
can be defined in a similar way.
$\mathcal{D}_v \Phi^L_t$ and $\mathcal{D}_v \Phi^H_t$ satisfies
the following two SPDEs respectively:
\begin{align}\label{e:low-Malliavin-derivative}
d\mathcal{D}_{v}\Phi^L
 + [A\mathcal{D}_{v}\Phi^L
    + D_L(B_L(\Phi,\Phi)\chi(\frac{|\Phi|_\W}{3\rho}))\mathcal{D}_{v}\Phi^L
    + D_H(B_L(\Phi,\Phi)\chi(\frac{|\Phi|_\W}{3\rho}))\mathcal{D}_{v}\Phi^H]\,dt = \notag\\
= [D_LQ_L(\Phi) \mathcal{D}_{v}\Phi^L+D_HQ_L(\Phi)\mathcal{D}_{v}\Phi^H]dW^L_t+Q_L(\Phi)v^L\,dt,
\end{align}
\begin{multline} \label{e:high-Malliavin-derivative}
 d\mathcal{D}_{v}\Phi^H +[A
\mathcal{D}_{v}\Phi^H+D_L(\mathrm{e}^{-A_H \delta}B_H(\Phi,\Phi)\chi
(\frac{|\Phi|_\W}{3\rho}))\mathcal{D}_{v}\Phi^L + \\
 +D_H(\mathrm{e}^{-A_H \delta}B_H(\Phi,\Phi)\chi
(\frac{|\Phi|_\W}{3\rho}))\mathcal{D}_{v}\Phi^H]\,dt
=Q_Hv^H\,dt
\end{multline}
with $\mathcal{D}_{v}\Phi_0^L=0$ and $\mathcal{D}_{v}\Phi_0^H=0$.
\ \\

Define the derivative flow of $\Phi^L(x)$ between $s$ and $t$ by $J_{s,t}(x)$,
$s\leq t$, which satisfies the following equation: for all $h\in H^L$
\begin{equation*}
dJ_{s,t}h+\Bigl[A J_{s,t}h+D_L[B_L(\Phi_t,\Phi_t)\chi
(\frac{|\Phi_t|_\W}{3\rho})]J_{s,t} h \Bigr]\,dt=D_LQ_L(\Phi_t)J_{s,t}hdW^L_t
\end{equation*}
with $J_{s,s}(x)=Id \in \mathcal{L}(H^L,H^L)$.
The inverse $J^{-1}_{s,t}(x)$ satisfies
\begin{equation} \label{e:J-1t}
dJ^{-1}_{s,t}h - J^{-1}_{s,t}\Bigl[A h+D_L[B_L(\Phi_t, \Phi_t)\chi (\frac{|\Phi_t|_\W}{3\rho})]h
-\tr((D_LQ_L(\Phi_t))^2)h \Bigr] \,dt
=-J^{-1}_{s,t} D_LQ_L(\Phi_t)hdW^L_t
\end{equation}
with $\tr((D_LQ_L(\Phi_t))^2)h=\sum_{k \in Z_L(N)} \sum_{i=1}^2 D [q_k(\Phi_t)e^i_k] D [q_k(\Phi_t)e_k^i] h$
and $q_k(x) = (1-\chi(|x|_\W/\rho))q_k$ (recall the notations in Appendix \ref{ss:geometry}). Simply writing $J_t=J_{0,t}$, clearly,
$$
J_{s,t}=J_t J^{-1}_s.
$$

We follow the ideas in Section 6.1 of \cite{EH} to develop a Malliavin
calculus for \eqref{e:NSE-delta}. One of the key points for this approach
is to find an adapted process $v \in L^2_{loc}(\mathbf{R}_{+}, H)$  so that
\begin{equation} \label{e:v direction}
Q_H v^H(t)
 = D_L(\mathrm{e}^{-A_H \delta}B_H(\Phi_t,\Phi_t)\chi(\frac{|\Phi_t|_\W}{3\rho}))\mathcal{D}_{v}\Phi_t^L,
\end{equation}
which implies that $\mathcal{D}_v\Phi^H_t=0$ for all $t>0$ (hence, the
Malliavin calculus is essentially restricted in \emph{low} frequency part).
More precisely,
\begin{prop} \label{prop:v-direction} 
There exists $v \in L^2_{\text{loc}}(\mathbf{R}_+; H)$ satisfying
\eqref{e:v direction}, and
\begin{equation*}
\mathcal{D}_{v} \Phi_t^L=J_t\int_0^t J^{-1}_{s} Q_L(\Phi_s)v^L(s)\,ds
\qquad\text{and}\qquad
\mathcal{D}_{v} \Phi_t^H=0.
\end{equation*}
\end{prop}
\begin{proof}
We first claim that
\begin{equation} \label{e:DeltaMotivation}
D_L(\mathrm{e}^{-A_H \delta}B_H(\Phi_t,\Phi_t)\chi
(\frac{|\Phi_t|_\W}{3\rho}))\mathcal{D}_{v}\Phi_t^L \in
(D(A^{\alpha_0+3/4}))^H.
\end{equation}
Indeed, $\Phi_t \in \widetilde{\W}$ from \eqref{e:W-tilde2}. Since $\mathcal{D}_v \Phi^L_t$
is finite dimensional,
$\mathcal{D}_v \Phi^L_t \in \widetilde{\W}$. It is easy to see
\begin{multline*}
D_L(\mathrm{e}^{-A_H \delta}B_H(\Phi_t,\Phi_t)\chi
(\frac{|\Phi_t|_\W}{3\rho}))\mathcal{D}_{v}\Phi_t^L =\mathrm{e}^{-A_H \delta}B_H(\mathcal{D}_v
\Phi^L_t,\Phi_t)\chi
(\frac{|\Phi_t|_\W}{3\rho})+\\
+\mathrm{e}^{-A_H \delta}B_H(\Phi_t,\mathcal{D}_v \Phi^L_t)\chi
(\frac{|\Phi|_\W}{3\rho})+\mathrm{e}^{-A_H \delta} B_H(\Phi_t, \Phi_t) \chi^{'}
(\frac{|\Phi_t|_\W}{3\rho})
\frac{\langle\Phi_t,\mathcal{D}_v \Phi^L_t\rangle_\W}{3\rho |\Phi_t|_\W}.
\end{multline*}
The three terms on the right hand of the above equality can all be bounded
in the same way, for instance, applying \eqref{e:B-uv1} with $\beta=\alpha_0+1/8$,
the first term is bounded by
\begin{equation*}
|\mathrm{e}^{-A_H \delta}B_H(\mathcal{D}_v \Phi^L_t,\Phi_t)\chi
(\frac{|\Phi|_\W}{3\rho})|_{D(A^{\alpha_0+\frac34})}=|A^{\frac78} \mathrm{e}^{-A_H \delta} A^{\alpha_0-\frac18} B_H(\mathcal{D}_v \Phi^L_t,\Phi_t)|_H \\
\leq \frac{C_1}{\delta^{\frac78}} |\mathcal{D}_v \Phi^L_t|_{\widetilde{\W}}|\Phi_t|_{\widetilde{\W}},
\end{equation*}
and \eqref{e:DeltaMotivation} follows immediately. Hence,
by Assumption [A3] for $Q$, there
exists at least one $v\in L^2_{loc}(\mathbf{R}_+;H)$ so that $v^H$ satisfies \eqref{e:v direction}
(we will see in \eqref{e:DvPhiL} that \emph{$\mathcal{D}_{v}\Phi_t^L$ does
not depend on $v^H$}). Thus equation \eqref{e:high-Malliavin-derivative}
is a homogeneous linear equation and has a unique solution
\[
\mathcal{D}_{v}\Phi_t^H=0,
\]
for all $t>0$. Hence, equation \eqref{e:low-Malliavin-derivative} now reads
\begin{equation*}
d\mathcal{D}_{v}\Phi^L+[A
\mathcal{D}_{v}\Phi^L+D_L(B_L(\Phi,\Phi)\chi
(\frac{|\Phi|_\W}{3\rho}))\mathcal{D}_{v}\Phi^L]\,dt=D_LQ_L(\Phi)
\mathcal{D}_{v}\Phi^LdW^L_t+Q_L(\Phi)v^L\,dt,
\end{equation*}
with $\mathcal{D}_{v}\Phi_0^L=0$, which is solved by
\begin{equation} \label{e:DvPhiL}
\mathcal{D}_{v} \Phi^L_t=\int_0^t J_{s,t} Q_L(\Phi_s)v^L(s)\,ds=J_t\int_0^t J^{-1}_{s} Q_L(\Phi_s)v^L(s) \,ds
\end{equation}
\end{proof}

Let $N\geq N_0$ be the integer fixed at the beginning of Section
\ref{sec:cutoff dynamics} and consider $M=2(2N+1)^3 - 2$ vectors
$v_1,\ldots,v_M\in L^2_{loc}(\mathbf{R}_+;H)$, with each of them
satisfying Proposition \ref{prop:v-direction} (notice that $M$ is the
dimension of $H^L=\pi_N H$). Set
\begin{equation} \label{e:Malliavin-direction}
v=[v_1, \ldots, v_M],
\end{equation}
we have
\begin{equation} \label{e:MalDerAloV}
\mathcal{D}_{v} \Phi_t^H=0,
\qquad
\mathcal{D}_{v} \Phi_t^L=J_t \int_0^t J^{-1}_s Q_L(\Phi_s)v^L(s) \,ds,
\end{equation}
where $Q_L$ is defined in \eqref{e:low and high frequency part equation}. Choose
\begin{equation*}
v^L(s)=(J^{-1}_s Q_L(\Phi_s))^{*}
\end{equation*}
and define the \emph{Malliavin matrix}
\begin{equation*}
\mathcal{M}_t=\int_0^t J^{-1}_s Q_L(\Phi_s)(J^{-1}_s Q_L(\Phi_s))^{*} \,ds.
\end{equation*}
Since $J^{-1}_t\in\mathcal{L}(\W^L,\W^L)$ and $Q_L \in \mathcal{L}(\W^L, \W^L)$,
it follows that $\mathcal{M}_t \in \mathcal{L}(\W^L,\W^L)$. By Parseval identity
(using the notation in Section \ref{ss:geometry}),
\begin{align} \label{e:representation of Malliavin}
\langle\mathcal{M}_t \eta, \eta\rangle>_{\W}
& = \int_0^t |(J^{-1}_s Q_L(\Phi_s))^*\eta|_{\W}^2\,ds
 = \sum_{k\in Z_L(N),i=1,2} \frac{1}{|k|^{4\alpha_0+1}}\int_0^t |\langle J^{-1}_s Q_L(\Phi_s)e^i_k, \eta\rangle_{\W}|^2\,ds\notag\\
& = \sum_{k\in Z_L(N),i=1,2} \frac{1}{|k|^{4\alpha_0+1}}\int_0^t |\langle J^{-1}_s (q_k(\Phi_s) e_k^i), \eta\rangle_{\W}|^2\,ds,
\end{align}
where $q_k(\Phi_s)=q_k(1-\chi(\tfrac{|\Phi_s|_\W}{\rho}))$ for $k\in Z_L(N_0)$
and $q_k(\Phi_s)=q_k$ for $k\in Z_L(N)\setminus Z_L(N_0)$.
\ \\

The following two lemmas are crucial for the proof of Lemma \ref{lem:Malliavin-estimate}.
The first one will be proven in the appendix (see page \pageref{pf:Jacobi-estimate}),
while the other in Section \ref{sub:Proof of Lemma}.
\begin{lem} \label{lem:Jacobi-estimate}
For any $T>0$ and $p \geq 2$, there exist some $C_i=C_i(p,\rho,\alpha_0)>0$ ($i=1,2,3,4$) such that
\begin{gather}
\E(\sup_{0 \leq t \leq T}|J_{t}(x)h^L|^p_\W) \leq C_1\mathrm{e}^{C_1 T}|h^L|^p_\W,  \label{e:Jacobi-estimate1}\\
\E(\sup_{0 \leq t \leq T}|J^{-1}_{t}(x)h^L|^p_\W) \leq C_2\mathrm{e}^{C_2 T}|h^L|^p_\W, \label{e:J-1}\\
\E(\sup_{0 \leq t \leq T}|J^{-1}_{t}(x)h^L-h^L|^p_\W) \leq T^{p/2}C_3 \mathrm{e}^{C_3T} |h^L|^p_\W,  \label{e:Jacobi-estimate2}\\
\E(\sup_{0 \leq t \leq T}|\Phi_t(x)-\mathrm{e}^{-At} x|^p_\W) \leq (T^{p/8} \vee T^{p/2}) C_4\mathrm{e}^{C_4T}.  \label{e:solution-estimate2}
\end{gather}

Suppose that $v_1$, $v_2$ satisfy Proposition \ref{prop:v-direction}
and $p \geq 2$, then
\begin{gather}
\E(\sup_{0 \leq t \leq T}|\mathcal{D}_{v_1}\Phi^L_t(x)|^p_\W) \leq C_5\mathrm{e}^{C_5T}\E[\int_0^T|v^L_1(s)|_\W^p \,ds]
	\label{e:Malliavin-one-order-estimate}\\
\E \Bigl(\sup_{0 \leq t \leq T}|\mathcal{D}^2_{v_1 v_2}\Phi^L_t(x)|^p_\W\Bigr)
	\leq  C_{6}\mathrm{e}^{C_6 T} \Bigl(\E[\int_0^T |v^L_1(s)|_\W^{2p} \,ds] \Bigr)^{1/2}
	\Bigl(\E[\int_0^T|v^L_2(s)|_\W^{2p} \,ds] \Bigr)^{1/2}
	\label{e:Malliavin-two-order-estimate}\\
\E\Bigl(\sup_{0 \leq t \leq T}|\mathcal{D}_{v_1}D_h\Phi^L_t(x)|^p_\W \Bigr)
	\leq  C_7\mathrm{e}^{C_7T}|h|_\W^{p} \Bigl(\E[\int_0^T|v^L_1(s)|_\W^{2p} \,ds] \Bigr)^{1/2}
	\label{e:mix-two-order-estimate}
\end{gather}
with $h \in \W$ and $C_i=C_i(p,\rho,\alpha_0)>0$, $i=5,6,7$.
\end{lem}
\begin{lem} \label{lem:inverse}
Suppose that $\Phi_t$ is the solution to equation \eqref{e:NSE-delta} with
initial data $x \in\widetilde{\W}$. Then $\mathcal{M}_t \in \mathcal{L}(\W^L,\W^L)$
is invertible almost surely. Denote $\lambda_{min}(t)$ the smallest eigenvalue of $\mathcal{M}_t$.
then there exists some $q>1$ (possibly large) such that for every $p>0$, there is some $C=C(p,\rho,\alpha_0)$ such that
\begin{equation} \label{e:MinEigVal}
P[|1/\lambda_{min}(t)| \geq 1/\epsilon^q] \leq \frac{C \epsilon^{p/8}(1+|x|_{\widetilde{\W}})^p}{t^p}
\end{equation}
\end{lem}
\ \\
Now let us combine the previous two lemmas to prove Lemma \ref{lem:Malliavin-estimate}.
\begin{proof}[Proof of Lemma \ref{lem:Malliavin-estimate}]
Under an orthonormal basis of $\W^L$, the operators $J_t$, $\mathcal{M}_t$, $\mcl D_v \Phi^L_t$ with $v$ defined in \eqref{e:Malliavin-direction}, and $D_L \Phi^L_t$ can all be represented by $M\times M$ matrices, where $M$ is the dimension of $\W^L$.
Let us consider
\[
\psi_{ik}(\Phi_t)=\psi(\Phi_t) \sum_{j=1}^M [(\mcl D_v \Phi^L_t )^{-1}]_{ij} [D_L \Phi^L_t]_{jk}
\qquad i, k=1, \ldots, M.
\]
Given any $h \in \W^L$, by \eqref{e:MalDerAloV}, it is easy to see that
\begin{equation} \label{e:Malliavin-inverse}
\begin{split}
D_L\psi_{ik}(\Phi_t) \mathcal{D}_v \Phi^L_t h
&=D_L \psi(\Phi_t) (\mathcal{D}_v \Phi^L_t h) \sum_{j=1}^M [(\mcl D_v \Phi^L_t )^{-1}]_{ij} [D_L \Phi^L_t]_{jk} \\
  &+\psi(\Phi_t) \sum_{j=1}^M
  \mathcal{D}_{vh} \left\{[(\mcl D_v \Phi^L_t )^{-1}]_{ij} [D_L \Phi^L_t]_{jk}\right\}
\end{split}
\end{equation}
where $v=v(t)$ is defined by \eqref{e:Malliavin-direction} with $v^L(t)=(J^{-1}_tQ_L(\Phi_t))^{*}$.
Note that $\W^L$ is isomorphic
to $\mathbf{R}^M$, given the standard orthonormal basis $\{h_i\;:\;i=1,\ldots,M\}$ of $\mathbf{R}^M$,
it can be taken as a presentation of the
orthonormal basis of $\W^L$.
Setting $h=h_i$ in \eqref{e:Malliavin-inverse}, summing over
$i$ and noticing the identity $\mathcal{D}_{v} \Phi^L_t=J_t \mathcal{M}_t$, we obtain
\begin{equation} \label{e:Malliavin-inverse 1}
\begin{split}
& \ \ \ \E \left(D_L \psi(X(t)) D_{h_k} \Phi^L_t\right) \\
&=\E \left(\sum \limits_{i=1}^{M}
\mathcal{D}_{vh_i} \psi_{ik}(\Phi_t)
\right)-\E \left(\sum \limits_{i,j=1}^{M} \psi(\Phi_t)
\mathcal{D}_{vh_i}\left\{[(\mcl D_v \Phi^L_t)^{-1}]_{ij} [D_L \Phi^L_t]_{jk}\right\}\right)
\end{split}
\end{equation}
Let us estimate the first term on the right hand of \eqref{e:Malliavin-inverse 1}
as follows. By Bismut formula and the identity $\mcl D_v \Phi^L_t=J_t \mcl M_t$ (see the argument
below \eqref{e:Malliavin-direction}),
\begin{equation} \label{e:FirTerMalEst}
\begin{split}
& \ \Bigl|\E \Bigl[\sum_{i=1}^M D_L\psi_{ik}(\Phi_t)\mathcal{D}_{v h_i} \Phi^L_t\Bigr]\Bigr| \\
& \leq \sum_{i,j=1}^M\Bigl|\E\Bigl[\psi(\Phi_t) [J_t^{-1} \mcl M^{-1}_t]_{ij} [D_L \Phi^L_t]_{jk} \int_0^t\langle v^L h_i,dW_s\rangle_H\Bigr]\Bigr| \\
& \leq ||\phi||_\infty \sum_{i,j=1}^M \E\left(\frac {1} {\lambda_{min}} |J^{-1}_t h_j| |D_{h_k} \Phi^L_t| |\int_0^t \langle v^L h_i,dW_s \rangle|\right),
\end{split}
\end{equation}
moreover, by H$\ddot{o}$lder's inequality, Burkholder-Davis-Gundy's inequality, \eqref{e:MinEigVal}, \eqref{e:J-1}, \eqref{e:Frechet-derivative} and the inequality (see $e_k^j$ in the appendix)
$$\E[|v^L(s) h_i|^2_\W]=\E[|(J^{-1}_s Q^L)^* h_i|_\W^2] \leq C \sum_{j=1}^M \sum_{k=1}^2 \E[|\langle h_i, J^{-1}_s Q^L e^j_k \rangle|^2] \leq Ce^{Ct}$$
in order, we have
\begin{equation} \label{e:EkJ-1Dhl}
\begin{split}
& \ \E\left(\frac {1} {\lambda_{min}} |J^{-1}_t h_j|_\W |D_{h_k} \Phi^L_t|_\W |\int_0^t
\langle v^L h_i,dW_s \rangle|\right) \\
& \leq \left [\E
\left(\frac{1}{\lambda_{min}^6}\right)\right]^{\frac16}
\left[\E \left(|J^{-1}_t h_j|_\W^6\right)\right]^{\frac16}
\left[\E \left(|D_{h_k} \Phi^L_t|_\W^6\right)\right]^{\frac16}
\left[\E(\int_0^t |(J^{-1}_sQ^L)^{*}h_i|^2 ds)\right]^{\frac 12} \\
& \leq \frac{Ce^{Ct}(1+|x|_{\widetilde{\W}})^p}{t^{p}}  \ \ \
\end{split}
\end{equation}
where $p>48q+1$ and $C=C(p,Q,\alpha_0, \rho)>0$. Combining \eqref{e:EkJ-1Dhl} and \eqref{e:FirTerMalEst}, one has
$$\Bigl|\E \Bigl[\sum_{i=1}^M D_L\psi_{ik}(\Phi_t)\mathcal{D}_{v h_i} \Phi^L_t\Bigr]\Bigr| \leq ||\phi||_\infty \frac{Ce^{Ct}(1+|x|_{\widetilde{\W}})^p}{t^p}.$$
By a similar argument but with more complicate calculation, we can have the same bounds for the second term on the r.h.s. of \eqref{e:Malliavin-inverse 1}. Hence,
\[
|\E \Bigl(D_L \psi(\Phi_t(x)) D_L \Phi^L_t(x) h_k \Bigr)|
\leq \frac{C_1 \mathrm{e}^{C_1 t} (1+|x|_{\widetilde{\W}})^p}{t^p}\|\psi\|_{\infty}
\]
where $C_1=C_1(p,\rho,\alpha_0,Q)>0$. Since the above
argument is in the framework of $\W^L$ with the orthonormal base
$\{h_k; 1 \leq k \leq M\}$, we have
\[
|\E \Bigl(D_L \psi(\Phi_t(x)) D_h \Phi^L_t(x)\Bigr)|
\leq \frac{C_1 \mathrm{e}^{C_1 t} (1+|x|_{\widetilde{\W}})^p}{t^p}\|\psi\|_{\infty} |h|_\W,
\]
for every $h \in \W^L$ and $t>0$.
\end{proof}

\subsection{H\"ormander's systems} 

This is an auxiliary subsection for the proof of Lemma \ref{lem:inverse}
given in the next subsection and we use the notations detailed in Section
\ref{ss:geometry} (in particular Subsection \ref{sss:noise}). Let
us consider the SPDE for $u^L$ in Stratanovich form as
\begin{equation} \label{e:LowFreqEqn}
du^L
   + [Au^L + B_L(u,u)\chi(\frac{|u|_\W}{3\rho})
   - \frac{1}{2}\sum_{\substack{k \in Z_L(N_0),\\i=1,2}} D_{q_k(u)e_k^i} q_k(u)e_k^i]\,dt
 = \sum_{\substack{k \in Z_L(N_0),\\i=1,2}} q_k(u)\circ dw_k(t)e_k
\end{equation}
where $q_k(u)=(1-\chi(\frac{|u|_\W}{\rho}))q_k$ for $k \in Z_L(N_0)$
and $q_k(u)=q_k$ for $k \in Z_L(N) \setminus Z_L(N_0)$. For any $x\in\W$,
it is clear that if $k\in Z_L(N_0)$ and $i=1,2$,
\[
D_{q_k(x)e_k^i} q_k(x)e_k^i
 = -\frac{1}{\rho}\chi'(\frac{|x|_\W}{\rho})\bigl(1-\chi(\frac{|x|_\W}{\rho})\bigr)\frac{\langle x,e_k^i\rangle_\W}{|x|_\W}.
\]
\ \\

 For any two Banach spaces $E_1$
and $E_2$, denote by $P(E_1,E_2)$ the set of all $C^\infty$ functions
$E_1\to E_2$ with all orders derivatives being polynomially bounded. If
$K\in P(H, H^L)$ and  $X\in P(H, H)$, define $[X,K]_L$ by
$$
[X,K]_L(x)=DK(x)X(x)-D_LX^L(x)K(x),
\qquad x\in H.
$$
For instance, $[A,K]_L\in P(D(A),H^L)$ with $[A,K]_L(x)=DK(x)Ax-A_LK(x)$.
Define
\[
X^0(x)
 =   Ax
   + \chi(\frac{|x|_\W}{3\rho})\mathrm{e}^{-\delta A_H}B(x,x)
   + \frac{1}{2\rho}\sum_{k\in Z_L(N_0),i=1,2} \chi'(\frac{|x|_\W}{\rho})\bigl(1-\chi(\frac{|x|_\W}{\rho})\bigr)\frac{\langle x,e_k^i\rangle_\W}{|x|_\W}e_k^i
\]
The brackets $[X^0,K]_L$ and $[A,K]_L$ will appear when \emph{applying
the It\^o formula on} $J_t^{-1}q^i_k(\Phi_t)$ (see \eqref{e:ItoJ}) in the
proof of Lemma \ref{lem:inverse}.

\begin{defn} \label{d:HorCon}
The H\"ormander's system $\mathbf{K}$ for equation \eqref{e:LowFreqEqn}
is defined as follows: given any $y\in\W$, define
\begin{align*}
&\mathbf{K}_0(y) = \{q_k(y)e_k^i\;:\;k\in Z_L(N),\; i=1,2\}\\
&\mathbf{K}_1(y) = \{[X^0(y),q_k(y)e_k^i]_L\;:\; k\in Z_L(N),\; i=1,2\} \\
&\mathbf{K}_2(y) = \{[q_k(y)e_k^i,K(y)]_L\;:\; K\in\mathbf{K}_1(y),\; k\in Z_L(N),\; i=1,2\}
\end{align*}
and $\mathbf{K}(y)=\mathbf{K}_0(y) \cup \mathbf{K}_1(y) \cup \mathbf{K}_2(y)$.
\end{defn}

\begin{prop} \label{prop:modified Hormander}
There exist $\overline{\rho}>0$ and $\overline{N}\geq N_0$ (which depend
only on $N_0$ and $Q$) such that if $\rho\geq\overline{\rho}$ and
$N\geq\overline{N}$, then the following property holds: for every
$x\in\W$ and  $h\in H^L$ there exist $\sigma>0$ and $R>0$ such that
\begin{equation} \label{con:Hormander}
\inf_{\delta>0} \sup_{K\in\mathbf{K}}\inf_{|y-x|_\W\leq R}|\langle K(y),h\rangle_\W|\geq \sigma |h|_\W.
\end{equation}
\end{prop}
\begin{proof}
We are going to show that there are $\sigma>0$ and $R>0$ (independent
of $\delta$) such that for every $x\in\W$ and $h\in\W^L$,
\begin{equation*}
\sup_{K\in\mathbf{K}}\inf_{|x-y|_\W\leq R} |\langle K(y),h\rangle_\W|\geq\sigma|h|_\W.
\end{equation*}
To this end, it is sufficient to show that there is a (finite) set
$\widetilde{\mathbf{K}}\subset\mathbf{K}(y)$ for every $y$, such that
$\Span(\widetilde{\mathbf{K}})=H^L$. We choose $R\leq\tfrac14\rho$.
\ \\

\underline{Case 1}: $|x|_\W\geq R+2\rho$. Hence $|y|_\W\geq2\rho$ for every $y$ such that
$|x-y|_\W\leq R$ and $q_k(y)=q_k$ for all $k$. So we can take $\widetilde{\mathbf{K}} = \mathbf{K}_0$
which spans the whole $H^L$ thanks to \eqref{e:LowCov}.
\ \\

\underline{Case 2}: $|x|\leq\rho - R$. Hence $|y|_\W\leq\rho$ for every $y$ such that $|x-y|_\W\leq R$
and $q_k(y) = 0$ for all $k\in Z_L(N_0)$. In particular, $X^0(y) = Ay + \mathrm{e}^{-\delta A_H}B(y,y)$
and so for $l$, $m\in Z_L(N)\setminus Z_L(N_0)$ and $i$, $j=1,2$
(cfr. Subsection \ref{sss:nonlin}),
\[
[q_l e_l^i, [X^0, q_m e_m^j]_L]_L
 = \pi_N B(q_l e_l^i, q_m e_m^j) + \pi_N B(q_m e_m^j, q_l e_l^i)
\]
(which are independent of $\delta$, thus providing the uniformity in
$\delta$ we need).
The proof that the vectors $[q_l e_l^i, [X^0, q_m e_m^j]_L]_L$, where
$l$, $m$ run over $Z_L(N)\setminus Z_L(N_0)$ and $i$, $j=1,2$, span $H^L$
follows exactly as in \cite{R1} (using \eqref{e:Bk1}-\eqref{e:Bk2}, since
the only difference is that here we use the Fourier basis \eqref{e:fourier}
rather than the complex exponentials).
 Hence, thanks to Lemma 4.2 of
\cite{R1}, it is sufficient to choose $N\geq N_0$ large enough so that
for every $k\in Z_L(N_0)$ there are $l$, $m\in Z_L(N)\setminus Z_L(N_0)$
such that $|l|\neq |m|$, $l$ and $m$ are linearly independent and $k=l+m$ (or $k=l-m$).
Take $\widetilde{\mathbf{K}}=\mathbf{K}_0 \cup\mathbf{K}_2$.
\ \\

\underline{Case 3}: $\rho - R\leq |x|_\W\leq 2\rho + R$, hence $|x|_\W\leq3\rho$ and
$|y|_\W\geq\tfrac12\rho$ for all $y$ such that $|x-y|_\W\leq R$. Write
$X^0(y) = X^{01}(y) + X^{02}(y)$ where $X^{01}(y) = Ay + \mathrm{e}^{-\delta A_H}B(y,y)$ and
\[
X^{02}(y) =
\frac{1}{2\rho}\sum_{k\in Z_L(N_0),i=1,2} \chi'(\frac{|y|_\W}{\rho})\bigl(1-\chi(\frac{|y|_\W}{\rho})\bigr)\frac{\langle y,e_k^1\rangle_\W}{|y|_\W} e_k^1.
\]
Choose $l$, $m\in Z_L(N)\setminus Z_L(N_0)$ and $i$, $j\in\{1,2\}$, then
\[
[q_l e_l^i, [X^0(y), q_m e_m^j]_L]_L
 = [q_l e_l^i, [X^{01}(y), q_m e_m^j]_L]_L + [q_l e_l^i, [X^{02}(y), q_m e_m^j]_L]_L.
\]
As in the previous case the vectors $[q_l e_l^i, [X^{01}(y), q_m e_m^j]_L]_L$
span the whole $H^L$, so, to conclude the proof we show that the other
term is a small perturbation. Indeed, $[q_l e_l^i, [X^{02}(y), q_m e_m^j]_L]_L$
corresponds to a derivative of $X^{02}$ in the directions $q_l e_l^i$
and $q_m e_m^j$ and it is easy to see by some straightforward computations that there is $c>0$, depending only on $N$,
$\chi$ and $Q$ (but not on $\rho$, $y$, $\delta$) such that
$|[q_l e_l^i, [X^{02}(y), q_m e_m^j]_L]_L|\leq\tfrac{c}{\rho^3}$.
So, for $\rho$ large enough, the vectors $[q_l e_l^i, [X^0(y), q_m e_m^j]_L]_L$
span $H^L$. Take $\widetilde{\mathbf{K}}=\mathbf{K}_0 \cup \mathbf{K}_2$.
\end{proof}
\ \\

\subsection{Proof of Lemma \ref{lem:inverse}} \label{sub:Proof of Lemma}
The key points for the proof are Proposition \ref {prop:modified Hormander}
and the following Norris' Lemma (Lemma 4.1 of \cite{No}).
\begin{lem}[Norris' Lemma]
Let $a, y \in \mathbb{R}$. Let
$\beta_t, \gamma_t=(\gamma^1_t, \ldots \gamma^m_t)$ and $u_t=(u^1_t,
\ldots, u^m_t)$ be adapted processes. Let
\begin{equation*}
a_t=a+\int_0^t \beta_s \,ds+\int_0^t \gamma^i_s dw^i_s,
\qquad
Y_t=y+\int_0^t a_s \,ds+\int_0^t u^i_s dw^i_s,
\end{equation*}
where $(w^1_t,\ldots,w^m_t)$ are i.i.d. standard Brownian motions.
Suppose that $T<t_0$ is a bounded stopping time such that for some
constant $C<\infty$:
\begin{equation*}
|\beta_t|, |\gamma_t|, |a_t|, |u_t| \leq C\qquad\text{for all }t\leq T.
\end{equation*}
Then for any $r>8$ and $\nu>\tfrac{r-8}{9}$ there is $C=C(T,q,\nu)$ such that
\begin{equation*}
P\Bigl[\int_0^T Y^2_t \,dt<\epsilon^r, \int_0^T (|a_t|^2+|u_t|^2)\,dt\geq\epsilon\Bigr]
<C\mathrm{e}^{-\frac{1}{\epsilon^{\nu}}}.
\end{equation*}
\end{lem}
\ \\

\begin{proof}[Proof of Lemma \ref{lem:inverse}]
We follow the lines of the proof of Theorem 4.2 of \cite{No}. Denote
$\mathcal{S}^L=\{\eta \in \W^L; |\eta|_{\W^L}=1\}.$ It is sufficient
to show the inequality \eqref{e:MinEigVal}, which is by \eqref{e:representation of Malliavin} equivalent to
\begin{equation}\label{e:malliavinbound}
P\Bigl[\inf_{\eta\in\mathcal{S}^L}
   \sum_{k\in Z_L(N),i=1,2} \frac{1}{|k|^{4\alpha_0+1}}\int_0^t |\langle J^{-1}_s q_k^i(\Phi_s), \eta\rangle_{\W}|^2\,ds
   \leq \epsilon^q \Bigr] \leq \frac{C\e^{p/8}(1+|x|_{\widetilde{\W}})^p}{t^p}
\end{equation}
for all $p>0$, where $q^i_k(\Phi_s)=q_k(\Phi_s)e^i_k$
with $q_k(\Phi_s)=q_k(1-\chi(\tfrac{|\Phi_s|_\W}{\rho}))$ for $k\in Z_L(N_0)$
and $q_k(\Phi_s)=q_k$ for $k\in Z_L(N)\setminus Z_L(N_0)$.
\ \\

Formula \eqref{e:malliavinbound} is implied by
\begin{equation}  \label{e:split}
D_\theta \sup_j \sup_{\eta \in \mathcal{N}_j}
   P\Bigl[\int_0^t\sum_{k \in Z_L(N),i=1,2}\frac{1}{|k|^{4\alpha_0+1}}|\langle J^{-1}_s q^i_k(\Phi_s),\eta\rangle_\W|^2\,ds
   \leq\epsilon^q\Bigr] \leq \frac{C \e^{p/8}(1+|x|_{\widetilde{\W}})^p}{t^p},
\end{equation}
for all $p>0$, where $\{\mathcal{N}_j\}_j$ is a finite sequence of disks
of radius $\theta$ covering $\mathcal{S}^L$, $D_{\theta}=\#\{\mathcal{N}_j\}$
and $\theta$ is sufficiently small. Define a stopping time $\tau$ by
\[
\tau=\inf \{s>0\;:\; |\Phi_s(x)-x|_\W>R,\;|J^{-1}_s-Id|_{\mathcal{L}(\W)}>c\}.
\]
where $R>0$ is the same as in \eqref{con:Hormander} and $c>0$ is sufficiently small. It is easy to see
that (\ref{e:split}) holds as long as for any $\eta \in \mathcal{S}^L$,
we have some neighborhood $\mathcal{N}(\eta)$ of $\eta$ and some
$k \in Z_L(N)$, $i \in \{1,2\}$ so that
\begin{equation} \label{e:qiksmall}
\sup_{\eta^{'} \in \mcl N(\eta)}P \Bigl(\int_0^{t \wedge
\tau} |\langle J^{-1}_s q^i_k(\Phi_s), \eta^{'}
\rangle_\W|^2\,ds \leq \epsilon^q \Bigr)=\frac{C \e^{p/8} (1+|x|_{\widetilde{\W}})^p}{t^p}.
\end{equation}
\ \\

The key point of the proof is to bound $P(\tau \leq \epsilon)$. By
\eqref{e:solution-estimate2} and the easy fact $|\mathrm{e}^{-At}x-x|_\W\leq C t^{1/8} |x|_{\widetilde{\W}}$,
we have for any $p \geq 2$
\begin{equation} \label{e:stopping time estimate 1}
\begin{split}
\E[\sup_{0 \leq t \leq T}|\Phi_t-x|^p_\W] & \leq \E[\sup_{0 \leq t \leq T} |e^{-At} x-x|_\W
+\sup_{0 \leq t \leq T} |\Phi_t(x)-e^{-At}x|_\W] \\
& \leq C_1 (1+|x|_{\widetilde{\W}})^p(T^{p/8} \vee T^{p/2})
\end{split}
\end{equation}
where $C_1=C_1(\alpha_0,p,\rho)$. Combining \eqref{e:stopping time estimate 1} and \eqref{e:Jacobi-estimate2}, we have
\begin{equation} \label{e:TauEst}
P(\tau \leq \epsilon)=C_1\epsilon^{p/8} (1+|x|_{\widetilde{\W}})^p
\end{equation}
for all $p>0$.  \\

Let us prove \eqref{e:qiksmall}. According to Definition \ref{d:HorCon} and Proposition \ref{prop:modified Hormander}, given a fixed $x \in \W$, for any $\eta \in \mathcal{S}^L$, there exists a $K \in {\bf K}$ such that
$$\sup_{K\in\mathbf{K}}\inf_{|y-x|_\W\leq R}|\langle K(y), \eta\rangle_\W|\geq \sigma |\eta|_\W.$$
Without loss of generality, assume that
$K \in {\bf K}_2$, so there exists some $q^i_k e_k$ and $q^j_l e_l$ such that
$$K_0(y):=q^i_k(y) e_k, \ K_1(y):=[X^0(y),q^i_k(y) e_k], \ K=K_2:=[q^j_l(y) e_l,K_1(y)].$$
Now one can follow the same but more simple argument as in Proof of Claim 2 in \cite{No} (page 127) to show that
\begin{equation*}
P \Bigl(\int_0^{t \wedge
\tau} |\langle J^{-1}_s q^i_k(\Phi_s), \eta^{'}
\rangle_\W|^2\,ds \leq \epsilon^{r^2} \Bigr)=\frac{C \e^{p/8} (1+|x|_{\widetilde{\W}})^p}{t^p},
\end{equation*}
(where the power $r^2$ is because one needs to use Norris' Lemma two times).  \\

Hence, take the neighborhood $\mathcal{N}(\eta)$ small enough and $q=r^2$,
 by the continuity, we have \eqref{e:qiksmall} immediately from the previous inequality.
\end{proof}

\section{Controllability and support}\label{s:Irreducibility}

The following proposition describes the support of the distribution
associated to a Markov solution.
\begin{prop} \label{p:support}
Let $(P_x)_{x\in H}$ be a Markov solution. For every $x\in\W$ and $T>0$,
the following properties hold,
\begin{itemize}
\item $P_x[\xi_T\in\W] = 1$,
\item for every $\W$-open set $U\subset\W$, $P_x[\xi_T\in U]>0$.
\end{itemize}
\end{prop}

The proof of the above proposition relies on the following control problem
(see~\cite{Shi06} for a general result on the same lines).
\begin{lem}\label{l:Control}
Given any $T>0$, $x,y\in\W$ and $\e>0$, there exist $\rho_0=\rho_0(|x|_\W,|y|_\W,T)$,
$u$ and $w$ such that
\begin{itemize}
\item $w\in L^2([0,T];H)$ and $u\in C([0,T];\W)$,
\item $u(0)=x$ and $|u(T) - y|_\W\leq\e$,
\item $\sup_{t\in[0,T]} |u(t)|_\W\leq\rho_0$,
\end{itemize}
and $u$, $w$ solve the following problem,
\begin{equation}\label{e:control}
\partial_t u + Au + B(u,u) = Qw,
\end{equation}
where $Q$ is defined in Assumption \ref{a:Q}.
\end{lem}
\begin{proof}
Let $z\in D(A^{\alpha_0+7/4})$ such that $|y-z|_\W\leq\tfrac{\e}2$, it
suffices to show that there exist $u,w$ satisfying the conditions of
the lemma and
\begin{equation} \label{e:ControlToZ}
|u(T)-z|_\W \leq\frac{\e}{2}.
\end{equation}
Decompose $u=u^H + u^L$  where $u^H=(I-\pi_{N_0})u$ and $u^L=\pi_{N_0} u$
and $N_0$ is the number in Assumption \ref{a:Q}, then equation~\eqref{e:control}
can be written as
\begin{align}
&\partial_t u^L + Au^L + B_L(u,u) = 0,\label{e:loweq}\\
&\partial_t u^H + Au^H + B_H(u,u) = Qw.\label{e:higheq}
\end{align}
We split $[0,T]$ into the pieces $[0,T_1]$, $[T_1,T_2]$, $[T_2,T_3]$ and $[T_3,T]$,
with the times $T_1$, $T_2$, $T_3$ to be chosen along the proof, and
prove that \eqref{e:ControlToZ} holds in the following four steps,
provided $\rho_0$ is chosen large enough (depending on $|x|_\W$, $|y|_\W$ and $T$).
\ \\

\noindent\emph{Step 1: regularization of the initial condition}.
Set $w\equiv0$ in $[0,T_1]$, using~\eqref{e:B-uv0}, one obtains
\begin{equation} \label{e:EnergyW}
\frac{d}{dt}|u|_\W^2 + 2|A^{\frac12}u|_\W^2
\leq 2|\langle A^{\frac34+\alpha_0}u,A^{\alpha_0-\frac14}B(u,u)\rangle_H|
\leq |A^{\frac12}u|_\W^2 + c |u|_\W^4.
\end{equation}
It is easy to see, by solving a differential inequality, that
$|u(t)|_\W^2+\int_0^t|A^{1/2}u|_\W^2\,ds\leq 2|x|_\W^2$ for
$t\leq t_0:=(2c|x|_\W^2)^{-1}$. In particular $u(t)\in D(A^{\alpha_0+3/4})$
for a.\ e.\ $t\in[0,t_0]$. An energy estimate similar to the one above, this time
in $D(A^{\alpha_0+3/4})$ and with initial condition $u(t_0/2)$ (w.l.o.g. assume $u(t_0/2) \in D(A^{\alpha_0+3/4})$), implies that
$u(t)\in D(A^{\alpha_0+5/4})$ a.\ e. for $t \in [t_0/2, t_0]$. By repeating the argument, we can finally
find a time $T_1\leq \tfrac{T}{4}\wedge t_0$ such that $u(T_1)\in D(A^{\alpha_0+7/4})$.
\ \\

\noindent\emph{Step 2: high modes led to zero}.
Choose a smooth function $\psi$ on $[T_1,T_2]$ such that $0\leq\psi\leq1$, $\psi(T_1)=1$
and $\psi(T_2)=0$, and set $u^H(t)=\psi(t)u^H(T_1)$ for $t\in[T_1,T_2]$.
An estimate similar to \eqref{e:EnergyW} yields
$$
\frac{d}{dt}|u^L|_\W^2 + |A^{\frac12}u^L|_\W^2
\leq c(|u^L|^2_\W + |u^H|_\W^2)^2,
$$
and $|u(t)|_\W^2\leq |u^L(t)|_\W^2 + |u^H(T_1)|^2_\W \leq 4 |x|^2_\W$ for
$T_1\leq t\leq T_2:=\tfrac{T}{2}\wedge(T_1 + (4c|x|_\W^2)^{-1})$.
Plug $u^L$ in~\eqref{e:higheq}, take
$$
w(t)
 =\psi'(t)Q^{-1}u^H(T_1) + \psi(t)Q^{-1}Au^H(T_1)+Q^{-1}B_H(u(t),u(t)).
$$
By the previous step $u(T_1) \in D(A^{\alpha_0+7/4})$,
$|Q^{-1}A u^H(T_1)|<\infty$; by \eqref{e:B-uv0},
$|Q^{-1}B_H(u(t),u(t))|\leq c|Au(t)|_\W^2\leq 2c N_0^4 (|Au^H(T_1)|_\W^2 + |u^L(t)|_\W^2)$
for $t\in[T_1,T_2]$. Hence, $w \in L^2([T_1,T_2],H)$.
\ \\

\noindent\emph{Step 3: low modes close to $z$}.
Let $u^L(t)$ be the linear interpolation between $u^L(T_2)$ and $z^L$
for $t\in[T_2,T_3]$. Write $u(t)=\sum u_k(t)e_k$, then \eqref{e:loweq}
in Fourier coordinates is given by
\begin{equation} \label{e:ControlLowPart}
\dot u_k + |k|^2 u_k + B_k(u,u) = 0,
\qquad k\in Z_L(N_0),
\end{equation}
where $B_k(u,u)=B_k(u^L,u^L)+B_k(u^L,u^H)+B_k(u^H,u^L)+B_k(u^H,u^H)$. Let us choose a
suitable $u^H$ to simplify the above $B_k(u,u)$. To this end, consider the set $\{(l_k,m_k): k \in Z_L(N_0)\}$
such that
\begin{enumerate}
\item If $k \in Z_L(N_0)_+$, then $l_k, -m_k \in Z_H(N_0)_+$ and $l_k+m_k=k$.
\item If $k \in Z_L(N_0)_-$, then $l_k, m_k \in Z_H(N_0)_+$ and $l_k-m_k=k$.
\item $|l_k| \neq |m_k|$ and $l_k \not\parallel m_k$ for all $k\in Z_L(N_0)$.
\item For every $k \in Z_L(N_0)$, $|l_k|,|m_k| \geq 2^{(2N_0+1)^3}$.
\item If $k_1 \neq k_2$, then $|l_{k_1} \pm l_{k_2}|, |m_{k_1} \pm m_{k_2}|, |l_{k_1} \pm m_{k_2}|, |m_{k_1} \pm l_{k_2}| \geq 2^{(2N_0+1)^3}$.
\end{enumerate}
\ \\

Define
\begin{equation*}
u^H(t)=\sum_{k \in Z_L(N_0)} u_{l_k}(t)e_{l_k} + u_{m_k}(t) e_{m_k},
\end{equation*}
with $u_{l_k}(t)$ and $u_{m_k}(t)$ to be determined by equation
\eqref{e:ControlXY} below. Using the formulas \eqref{e:Bk1}-\eqref{e:Bk2} in Section \ref{sss:nonlin}, it is easy to see
that
\begin{itemize}
\item by (4), $B_k(u^L,u^H) = B_k(u^H,u^L) = 0$,
\item by (5), $B_k(u_{l_{k_1}}, u_{l_{k_2}})=B_k(u_{l_{k_1}}, u_{m_{k_2}})=B_k(u_{m_{k_1}}, u_{l_{k_2}})=B_k(u_{m_{k_1}}, u_{m_{k_2}})=0$.
\end{itemize}
Hence, using again the computations
of Section \ref{sss:nonlin}, equation \eqref{e:ControlLowPart} is
simplified to the following equation
\begin{equation} \label{e:ControlXY}
\begin{cases}
(m_k \cdot X)\mathcal{P}_kY \pm (l_k \cdot Y) \mathcal{P}_kX + 2G_k(t) = 0,\\
X\cdot l_k=0,\quad Y\cdot m_k = 0,\quad l_k \pm m_k=k,
\end{cases}
\end{equation}
for each $k\in Z_L(N_0)_\pm$, where $G_k=\dot u_k + |k|^2 u_k + B_k(u^L,u^L)$
is a polynomial in $t$ and clearly $G_k\cdot k = 0$.
In order to see that the above equation has a solution, consider for
instance the case $k\in Z_L(N_0)_+$. Let $\{\vec{k}, g_1, g_2\}$ be an
orthonormal basis of $\mathbf{R}^3$ such that $l_k$, $m_k\in\Span(\vec{k}, g_1)$,
and $\vec{k}=\tfrac{k}{|k|}$. Let $X = x_0\vec{k} + x_1 g_1 + x_2 g_2$
and $Y = y_0\vec{k} + y_1 g_1 + y_2 g_2$.
A simple computation yields
$$
(X\cdot m_k)(\mathcal{P}_k Y) + (Y\cdot l_k)(\mathcal{P}_k X)
 = |k|(x_0 y_2 + x_2 y_0)g_2 - |k| c_k x_0 y_0 g_1,
$$
where $c_k = \tfrac{|l_k|^2-|m_k|^2}{\sqrt{|l_k|^2|m_k|^2-(l_k\cdot m_k)^2}}$.
One can for instance set $x_0=1$, $x_2=1$ and solve the problem in the
unknown $y_0$, $y_2$ (notice that $x_1$, $y_1$ can be determined by the
divergence free constraint).
\ \\

In conclusion the solution $u^H(t)$ is smooth in $t$ and by this construction
the dynamics $u = u^L + u^H$ is finite dimensional. Hence $u(t)$ is
smooth in space and time for $t\in [T_2, T_3]$ and $\sup|u(t)|_\W$ can
be bounded only in terms of $|u^L(T_2)$, $z^L$ and $T_3-T_2$.
We finally set $w=Q^{-1}[\dot u^H + Au^H + B_H(u,u)]$.
\ \\

\noindent\emph{Step 4: high modes close to $z$}.
In the interval $[T_3,T]$ we choose $u^H$ as the linear interpolation
between $u^H(T_3)$ and $z^H$. Let $u^L$ be the solution to equation
\eqref{e:loweq} on $[T_3,T]$ with the choice of $u^H$ given above.
Since $u(T_3)\in D(A^{\alpha_0+7/4})$ and $u^L(T_3)=z^L$ from step 3, by the
continuity of the dynamics, $\sup_{T_3 \leq t \leq T}|u^L(t)-z^L|_\W\leq \frac{\e}2$
if $T-T_3$ is small enough (recall that we can choose an arbitrary $T_3 \in (T_2,T)$
in the third step). Thus \eqref{e:ControlToZ} holds and, as in the
second step, we can find $w\in L^2([T_3,T],H)$ solving \eqref{e:higheq}.
It is clear from the above construction that
$\sup_{T_3 \leq t \leq T} |u(t)|_\W \leq C|z|_{\W}+C|u(T_3)|_\W$.
\end{proof}

\begin{proof}[Proof of Proposition \ref{p:support}]
The first property follows from Theorem 6.3 of \cite{FR} (which
only uses strong Feller). For the second property, fix $x\in\W$
and $T>0$, then it is sufficient to show that for every $y\in\W$
and $\e>0$, $P_x[|\xi_T-y|_\W\leq\e]>0$. Consider
$\rho>\rho_0$ (where $\rho_0$ is the constant provided by Lemma
\ref{l:Control}), then by Theorem~\ref{t:weakstrong},
$$
P_x[|\xi_T-y|_\W\leq\e]
\geq P_x[|\xi_T-y|_\W\leq\e,\;\tau_\rho>T]
 =   P_x^\rho[|\xi_T-y|_\W\leq\e,\;\tau_\rho>T].
$$
By Lemma~\ref{l:Control} there exist $\overline{\eta}$ and $\overline{u}$
such that $\overline{u}$ is the solution to the control problem \eqref{e:control}
connecting $x$ at $0$ with $y$ at $T$ corresponding to the control
$\partial_t\overline{\eta}$. Choose $s\in(0,\frac12)$, $p>1$ and
$\beta>\tfrac34$ such that $s-\tfrac1p>0$ and $\beta+\tfrac1{p}-s<\tfrac12$,
then by Lemma $C.3$ of \cite{FR} (which does not rely on non-degeneracy
of the covariance), there is $\delta>0$ such that for all $\eta$
in the $\delta$-ball $B_\delta(\overline{\eta})$ centred at
$\overline{\eta}$ in $W^{s,p}([0,T];D(A^{-\beta}_H))$, we have
that $|u(T,\eta)-y|_\W\leq\e$ and $\sup_{[0,T]}|u(t,\eta)|_\W\leq\rho_0$,
where $u(\cdot,\eta)$ is the solution to the control problem with control
$\partial_t \eta$. By proceeding as in the proof of Proposition
6.1 of \cite{FR}, it follows that in conclusion the probability
$P_x^\rho[|\xi_T-y|_\W\leq\e,\;\tau_\rho>T]$ is bounded from below
by the (positive) measure of $B_\delta(\overline{\eta})$ with respect
to the Wiener measure corresponding to the cylindrical Wiener process
on $H$.
\end{proof}

\appendix
\section{Appendix}

\subsection{Details on the geometry of modes}\label{ss:geometry}

Here we reformulate the problem in Fourier coordinates and explain
in full details the conditions of Assumption \ref{a:Q}.
\ \\

Define $\Z^3_*=\Z^3\setminus\{(0,0,0)\}$,
$\Z^3_+=\{k\in\Z^3: k_1>0 \}\cup\{k\in\Z^3: k_1=0, k_2>0\}\cup\{k\in\Z^3; k_1=0, k_2=0, k_3>0\}$
and $\Z^3_-=-\Z^3_+$, and set
\begin{equation}\label{e:fourier}
e_k(x) =
\begin{cases}
\cos k\cdot x &\qquad k\in\Z^3_+,\\
\sin k\cdot x &\qquad k\in\Z^3_-.
\end{cases}
\end{equation}
Fix for every $k\in\Z^3_*$ an arbitrary orthonormal basis $(x_k^1,x_k^2)$ of
the subspace $k^\perp$ of $\mathbf{R}^3$ and set $e_k^1 = x_k^1 e_k(x)$ and
$e_k^2 = x_k^2 e_k(x)$, then $\{ e_k^i : k\in\Z^3_*,\;i=1,2\}$ is an
\emph{orthonormal basis} of $H$. In particular,
$\pi_N H = \Span(\{e_k^i:0<|k|_\infty\leq N,\;i=1,2\})$.
Denote moreover, for any $N>0$, $Z_L(N)=[-N,N]^3\setminus (0,0,0)$
and $Z_H(N)=\Z^3_*\setminus Z_L(N)$.

\subsubsection{Assumptions on the covariance}\label{sss:noise}
Under the Fourier basis of $H$, the \emph{diagonality} assumption [A1] means that
for each $k \in\Z_+^3$, there exists some linear operator
$q_k:k^{\bot} \rightarrow k^{\bot}$ such that $Q(y e_k)=(q_k y) e_k$ for
$y \in k^{\bot}$.
The \emph{finite degeneracy} assumption [A2] says that
$q_k$ is invertible on $k^\perp$  if $k \in Z_H(N_0)$ and $q_k=0$ otherwise.
If $W$ is a cylindrical
Wiener process on $H$, then $Q\,dW = \sum_{k\in Z_H(N_0)} e_k q_k \,dw_k$,
where $(w_k)_{k\in Z_H(N_0)}$ is a sequence of independent 2d Brownian
motions and each $w_k \in k^\perp$.
\ \\

The $\overline{Q}$ in \eqref{e:NSE-cutoff} is a non-degenerate operator on $\pi_{N_0}H$, which is defined under the Fourier basis by
\begin{equation} \label{e:LowCov}
\overline Q=\sum_{k \in Z_L(N_0)} e_k q_k \langle \cdot, e_k \rangle_H,
\end{equation}
where, for each $k\in Z_L(N_0)$, $q_k$ is an invertible operator on $k^\perp$.
\subsubsection{The nonlinearity}\label{sss:nonlin}

In Fourier coordinates, equation \eqref{e:NSEabs} can be represented
under the Fourier basis by
\begin{equation*}
\begin{cases}
du_k + [|k|^2 u_k + B_k(u,u)]\,dt = q_k\, dw_k(t),
   \qquad k \in Z_H(N_0)\\
du_k + [|k|^2 u_k + B_k(u,u)]\,dt = 0,
   \qquad k\in Z_L(N_0)\\
u_k(0)=x_k, \qquad k\in\Z^3_*,
\end{cases}
\end{equation*}
where $u = \sum u_k e_k$, $u_k\in k^\perp$ for all $k\in\Z^3_*$
and $B_k(u,u)$ is the Fourier coefficient of $B(u,u)$ corresponding
to $k$. To be more precise,
$$
B(u,u)
 = \sum_{l, m\in\Z^3_*} B(u_le_l, u_me_m)$$
and if, for instance, $l$, $-m$, $l+m\in\Z^3_+$,
\begin{equation*}
B(u_le_l, u_me_m) = \mathcal{P}\bigl((u_l\cdot m)u_m e_l e_{-m}\bigr)
= \frac 12 \bigl[(u_l\cdot m)\mathcal{P}_{l+m}u_m e_{l+m} + (u_l\cdot m)\mathcal{P}_{l-m}u_m e_{l-m} \bigr],
\end{equation*}
where $\mathcal{P}_k$ is the projection of $\mathbf{R}^3$ onto $k^\perp$,
given by $\mathcal{P}_k\eta=\eta-\tfrac{k\cdot\eta}{|k|^2}k$, then, clearly,
\begin{align}
B_{l+m}(u_le_l, u_me_m) &= \tfrac 12 (u_l\cdot m)\mathcal{P}_{l+m}u_m e_{l+m}, \label{e:Bk1} \\
B_{l-m}(u_le_l, u_me_m) &= \tfrac 12 (u_l\cdot m)\mathcal{P}_{l-m}u_m e_{l-m},  \label{e:Bk2}
\end{align}
and $B_k(u_le_l, u_me_m)=0$ otherwise. For the other cases (of $l, m$), similar
formulas hold.

\subsection{Proofs of the auxiliary results}

The key points for the proofs of this section are the following two inequalities
and Lemma \ref{lem:noise-estimate} below.
Given $\beta>\tfrac12$, there exist constants $C_1>0$, $C_2>0$ such that
for every $u$, $v\in D(A^{\beta+1/4})$,
\begin{gather}
|A^{\beta-\frac14}B(u,v)|_H\leq C_1 |A^{\beta+\frac14}u|_H |A^{\beta+\frac14}v|_H, \label{e:B-uv0}\\
|A^{\beta+\frac14}\mathrm{e}^{-At}B(u,v)|_H\leq \frac{C_2}{\sqrt{t}}|A^{\beta+\frac14}u|_H |A^{\beta+\frac14}v|_H. \label{e:B-uv1}
\end{gather}
The first inequality is given by Lemma D.2. in \cite{FR}, the
second follows from the standard estimate $|A^{1/2}\mathrm{e}^{-At}|_H\leq Ct^{-1/2}$
for analytical semigroups. The other basic tool is the following Lemma
which is a straightforward modification of Proposition 7.3 of~\cite{DPZ}.

\begin{lem}\label{lem:noise-estimate}
Let $Q:H\to H$ be a linear bounded operator such that $A^{\alpha_0+3/4}Q$
is also bounded, and let $W$ be a cylindrical Wiener process on $H$.
Then for any $0<\beta<\tfrac14$, $p>2$ and $\epsilon\in [0,\tfrac14-\beta)$,
there exists $C>0$ such that
\begin{equation*}
\E\Bigl[\sup_{0\leq t\leq T} |A^\beta\int_0^t \mathrm{e}^{-A(t-s)}Q\,dW_s|_\W^p\Bigr]
\leq C T^{(\frac14-\epsilon-\beta)p} |A^{-\frac34-\epsilon}|^p_{HS}.
\end{equation*}
\end{lem}

\begin{proof}[Proof of Lemma \ref{lem:low-high-freqency-estimate}]\label{pf:low-high-freqency-estimate}
We simply write $\Phi_t=\Phi^{\delta}_t$ (with $\delta \geq 0$) and prove \eqref{e:Frechet-derivative-integral}
at the end. Clearly, $\Phi_t(x)$ satisfies the following equation
\begin{equation*}
\Phi_t=\mathrm{e}^{-At} x+ \int_0^t \mathrm{e}^{-A(t-s)} \mathrm{e}^{-A_H \delta} B(\Phi_s, \Phi_s)\chi
(\frac{|\Phi_s|_\W}{3\rho})\,ds+\int_0^t \mathrm{e}^{-A(t-s)}
Q(\Phi_s)\,dW_s.
\end{equation*}
By inequality \eqref{e:B-uv1}, the fact $|\mathrm{e}^{-A_H \delta}|_\W\leq 1$ and
the inequality $\chi(\frac{|\Phi_t|_\W}{3 \rho}) |\Phi_t|_\W \leq 3 \rho$,
it is easy to see that
\begin{align*}
|\Phi_t|_\W
& \leq |x|_\W
  + \int_0^t|\mathrm{e}^{-A(t-s)} B(\Phi_s, \Phi_s)|_\W \chi(\frac{|\Phi_s|_\W}{3\rho})\,ds
  + |\int_0^t\mathrm{e}^{-A(t-s)}Q(\Phi_s)\,dW_s|_\W \\
& \leq |x|_\W
  + \int_0^t \frac{C \rho}{\sqrt{t-s}}|\Phi_s|_\W \cdot \chi(\frac{|\Phi_s|_\W}{3\rho})\,ds
  + |\int_0^t \mathrm{e}^{-A(t-s)}(1-\chi(\frac{|\Phi_s|_\W}{\rho})) Q\,dW_s|_\W \\
& \leq |x|_\W
  + C \rho t^{\frac12}\sup_{0 \leq s \leq t} |\Phi_s|_\W
  + |\int_0^t \mathrm{e}^{-A(t-s)}(1-\chi(\frac{|\Phi_s|_\W}{\rho})) Q\,dW_s|_\W,
\end{align*}
and that for any $p \geq 2$, $T>0$,
\begin{equation*}
\E\Bigl(\sup_{0 \leq t \leq T}|\Phi_t|^p_\W\Bigr)
\leq |x|^p_\W + C_1 T^{p/8}+C_1T^{p/2} \E\Bigl(\sup_{0 \leq t \leq T}|\Phi_t|^p_\W \Bigr)
\end{equation*}
by Lemma \ref{lem:noise-estimate} (with $\epsilon=\tfrac18$, $\beta=0$) and some basic computation,
with $C_1=C_1(p,\alpha_0,\rho)$. For $T$ small,
$\E(\sup_{0 \leq t \leq T}|\Phi_t|^p_\W)\leq \tfrac{|x|^p_\W+C_1T^{p/8}}{1-C_1T^{p/2}}$.
Now, by taking $T,2T,\ldots$ as initial times, by applying the same
procedure on $[T,2T],[2T,3T],\ldots$, respectively one can obtain
similar estimates as the above on these time intervals.
Inductively, the estimate \eqref{e:solution-estimate1} follows.
The proof of \eqref{e:W-tilde1} and \eqref{e:W-tilde2} proceeds similarly.

For every $h\in\W$, $D_h \Phi_t$ satisfies the following equation
\begin{multline*}
D_h \Phi_t
 = \mathrm{e}^{-At} h
  + \int_0^t \mathrm{e}^{-A(t-s)} (B(D_h\Phi_s,\Phi_s)+B(\Phi_s, D_h\Phi_s)) \chi(\frac{|\Phi_s|_\W}{3\rho})+ {}\\
  + \mathrm{e}^{-A(t-s)}B(\Phi_s, \Phi_s) \chi'(\frac{|\Phi_s|_\W}{3\rho}) \frac{1}{3\rho}
    \cdot\frac{\langle D_h\Phi_s,\Phi_s\rangle_\W}{|\Phi_s|_\W}\,ds + {}\\
  - \int_0^t \mathrm{e}^{-A(t-s)} \chi'(\frac{|\Phi_s|_\W}{\rho}) \frac{1}{\rho} \cdot
    \frac{\langle D_h\Phi_s,\Phi_s\rangle_\W}{|\Phi_s|_\W}Q_L\,dW^L_s,
\end{multline*}
By \eqref{e:B-uv1} and  $\chi(\frac{|\Phi_t|_\W}{3\rho})|\Phi_t|_\W \leq 3\rho$,
\begin{align*}
|D_h \Phi_t|_\W
& \leq |h|_\W + \int_0^t \frac{C}{\sqrt{t-s}} \Bigl(\chi(\frac{|\Phi_s|_\W}{3\rho})|\Phi_s|_\W +
                         \frac{1}{3\rho}|\Phi_s|_\W^2 |\chi' (\frac{|\Phi_s|_\W}{3\rho})| \Bigr) |D_h \Phi_s|_\W\,ds \\
&\quad + \frac{1}{\rho} \Bigl|\int_0^t \mathrm{e}^{-A(t-s)} \chi'(\frac{|\Phi_s|_\W}{\rho}) \frac{\langle D_h\Phi_s,\Phi_s\rangle_\W}{|\Phi_s|_\W}Q_L\,dW^L_s\Bigr|_\W \\
& \leq |h|_\W
  + \int_0^t \frac{C \rho}{\sqrt{t-s}}|D_h\Phi_s|_\W\,ds
  + \frac{1}{\rho} \Bigl|\int_0^t \mathrm{e}^{-A(t-s)} \chi'(\frac{|\Phi_s|_\W}{\rho})
                         \frac{\langle D_h\Phi_s,\Phi_s\rangle_\W}{|\Phi_s|_\W}Q_L\,dW^L_s\Bigr|_\W,
\end{align*}
by Lemma \ref{lem:noise-estimate} (with $\beta=0$ and
$\epsilon=\tfrac18$),
\begin{align*}
\E\Bigl[\sup_{0 \leq t \leq T} |D_h \Phi_t|_\W^p\Bigr]
\leq |h|_\W^p+ C T^{\frac{p}8} \E\Bigl[\sup_{0 \leq t\leq T}|D_h\Phi_t|_\W^p\Bigr],
\qquad 0\leq T\leq 1,
\end{align*}
where $C=C(\alpha_0, p, \rho)>0$. For $T>0$ small enough,
$\E [\sup_{0 \leq t \leq T} |D_h \Phi_t|^p] \leq\frac{1}{1-C T^{p/8}}|h|_\W^p$.

For $|D_{h^H} \Phi^L_t|_\W$, it is easy to see by a similar argument as in
proving \eqref{e:Frechet-derivative} that
\[
|D_{h^H} \Phi^L_t|_\W
\leq \int_0^t \frac{C \rho}{\sqrt{t-s}}|D_{h^H}\Phi_s|_\W\,ds
+ \frac{1}{\rho}\Bigl|\int_0^t \mathrm{e}^{-A(t-s)} \chi'(\frac{|\Phi_s|_\W}{\rho})
\frac{\langle D_{h^H}\Phi_s,\Phi_s\rangle_\W}{|\Phi_s|_\W}Q_L\,dW^L_s\Bigr|_\W,
\]
so by Lemma \ref{lem:noise-estimate} and \eqref{e:Frechet-derivative},
\begin{align*}
& \E\Bigl[\sup_{0 \leq t \leq T}|D_{h^H} \Phi^L_t|^p_\W\Bigr]
\leq T^{\frac{p}8} C\mathrm{e}^{C T}|h^H|^p_\W,\qquad 0\leq T \leq 1, \\
& \E\Bigl[\sup_{0 \leq t \leq T}|D_{h^H} \Phi^L_t|^p_\W]
\leq T^{\frac{p}2}C  \mathrm{e}^{C T}|h^H|^p_\W,\qquad T>1,
\end{align*}
where $C=C(\alpha_0, p, \rho)>0$. Similarly but more simply, we have \eqref{e:DLH}.

Let us now prove \eqref{e:Frechet-derivative-integral}. By It\^o formula,
\begin{multline*}
\E|D_h\Phi_t|^2_\W + 2\int_0^t \E|A^{\frac12}D_h\Phi_s|^2_\W\,ds\leq \\
\leq |h|^2_\W + C \rho \int_0^t \E \Bigl[|A^{\frac12}D_h\Phi_s|_\W|A^{\alpha_0-\frac14} D_h[\mathrm{e}^{-A_H \delta} B(\Phi_s,\Phi_s)
  \chi(\frac{|\Phi_s|_\W}{3\rho})]|_H \Bigr]\,ds.
\end{multline*}
By \eqref{e:B-uv0} and Cauchy inequality, we have
\begin{equation*}
\E|D_h \Phi_t|^2_\W+\int_0^t \E|A^{\frac12}D_h \Phi_s|^2_\W\,ds \leq |h|^2_\W + C\int_0^t \E|D_h \Phi_s|_\W^2\,ds
\end{equation*}
with $C=C(\alpha_0, \rho)>0$, which easily implies \eqref{e:Frechet-derivative-integral}
by Gronwall's lemma.
\end{proof}


\begin{proof}[Proof of Proposition \ref{prop:approximate cutoff}]
Recall that the solutions to \eqref{e:NSE-cutoff} and \eqref{e:NSE-delta}
are respectively denoted by $\Phi_t(x)$ and $\Phi_t^\delta(x)$.
 Denote $\Psi_t=\Phi_t-\Phi^\delta_t$, we have
\begin{equation} \label{e:approximate cutoff 1}
\Psi_t=\int_0^t I_1 \,ds+\int_0^t I_2 \,dW_s
\end{equation}
with
$$
I_1=\mathrm{e}^{-A(t-s)}[B(\Phi_s, \Phi_s) \chi
(\frac{|\Phi_s|_\W}{3\rho})-\mathrm{e}^{-A \delta}B(\Phi^{\delta}_s, \Phi^{\delta}_s) \chi
(\frac{|\Phi^{\delta}_s|_\W}{3\rho})], \qquad
$$
and $I_2=\mathrm{e}^{-A(t-s)}[Q(\Phi_s)-Q(\Phi^{\delta}_s))]$.
By \eqref{e:B-uv1},
\begin{equation}  \label{e:approximate cutoff 2}
\begin{split}
|I_1|_\W
& \leq |Id-\mathrm{e}^{-A \delta}|_{\mathcal{L}(\W)}|\mathrm{e}^{-A(t-s)}B(\Phi_s, \Phi_s)|_\W\chi
(\frac{|\Phi_s|_\W}{3\rho}) \\
&\quad +\Bigl|\mathrm{e}^{-A(t-s)}B(\Phi_s, \Phi_s) \chi
(\frac{|\Phi_s|_\W}{3\rho})-\mathrm{e}^{-A(t-s)}B(\Phi^{\delta}_s, \Phi^{\delta}_s) \chi
(\frac{|\Phi^{\delta}_s|_\W}{3\rho})\Bigr|_\W \\
&\leq \frac{C_1}{\sqrt{t-s}} |Id-\mathrm{e}^{-A \delta}|_{\mathcal{L}(\W)} + \frac{C_2}{\sqrt{t-s}} |\Psi_s|_\W
\end{split}
\end{equation}
with $C_1=C_1(\rho, \alpha_0)$ and $C_2=C_2(\rho, \alpha_0)$, since
\begin{align*}
&\Bigl|\mathrm{e}^{-A(t-s)}B(\Phi_s, \Phi_s) \chi(\frac{|\Phi_s|_\W}{3\rho})-\mathrm{e}^{-A(t-s)}B(\Phi^{\delta}_s, \Phi^{\delta}_s) \chi
(\frac{|\Phi^{\delta}_s|_\W}{3\rho})\Bigr|_\W \\
&\quad=\Bigl|\int_0^1  \mathrm{e}^{-A(t-s)} \frac{d}{d\lambda} [B(\lambda \Phi_s+(1-\lambda) \Phi^{\delta}_s, \lambda \Phi_s+(1-\lambda) \Phi^{\delta}_s) \chi
(\frac{|\lambda \Phi_s+(1-\lambda) \Phi^{\delta}_s|_\W}{3\rho})] d\lambda \Bigr|_\W \\
&\quad\leq \frac{C_2}{\sqrt{t-s}} |\Psi_s|_\W
\end{align*}
By fundamental calculus and Lemma \ref{lem:noise-estimate} (with $\beta=0$ and $\epsilon=1/8$),
\begin{align} \label{e:approximate cutoff 3}
\E\Bigl[\sup_{0 \leq t \leq T} |\int_0^t I_2\,dW_s|^p \Bigr]
&\leq \E\Bigl[\sup_{0 \leq t \leq T} |\int_0^t \mathrm{e}^{-A(t-s)}(\chi
(\frac{|\Phi_s|_\W}{\rho})-\chi
(\frac{|\Phi^{\delta}_s|_\W}{\rho})) Q_L\,dW^L_s|^p \Bigr] \notag\\
& \leq \E \Bigl[\int_0^1 \sup_{0 \leq t \leq T} |\int_0^t \mathrm{e}^{-A(t-s)} \frac{d}{d \lambda} \chi
(\frac{|\lambda \Phi_s+(1-\lambda) \Phi^{\delta}_s|_\W}{\rho})  Q_L\,dW^L_s|^p d\lambda\Bigr] \\
&\leq C_3 T^{p/2} \E\Bigl[\sup_{0 \leq t \leq T} |\Psi_t|^p_\W \Bigr],\notag
\end{align}
with $p \geq 2$, $C_3=C_3(p, \alpha_0, \rho)$ and $T>0$.
Combining \eqref{e:approximate cutoff 1}, \eqref{e:approximate cutoff 2} and \eqref{e:approximate cutoff 3}, we have
\begin{equation} \label{e:approximate cutoff 4}
\E\Bigl[\sup_{0 \leq t \leq T} |\Psi_t|^p_\W \Bigr] \leq C_1 T^{\frac{p}2}
|Id-\mathrm{e}^{-A \delta}|^p_{\mathcal{L}(\W)}+C_4 T^{\frac{p}{2}} \E\Bigl[\sup_{0 \leq t \leq T} |\Psi_t|^p_\W \Bigr]
\end{equation}
with $C_4=C_4(p,\alpha_0, \rho)>0$. With the estimate of \eqref{e:approximate cutoff 4} and
by the same induction argument as in the proof of Lemma \ref{lem:low-high-freqency-estimate},
estimate \eqref{e:ApproxCutoff1} follows.

As for the estimate \eqref{e:ApproxCutoff2}, differentiating both sides of \eqref{e:approximate cutoff 1}
along directions $h \in\W$, applying the same method as above but with a little more complicated computation,
and noticing \eqref{e:Frechet-derivative}, we have
\begin{equation*}
\E\Bigl[\sup_{0 \leq t \leq T} |D_h\Psi_t|^p_\W \Bigr] \leq C_5 \mathrm{e}^{C_5 T}|Id-\mathrm{e}^{-A \delta}
|^p_{\mathcal{L}(\W)} |h|^p_\W,
\end{equation*}
for all $h\in \W$, with $C_5=C_5(\alpha_0, \rho, p)$. Formula \eqref{e:limitPS} follows
from the two estimates in the lemma immediately.
\end{proof}

\begin{proof}[Proof of Lemma \ref{lem:Jacobi-estimate}]\label{pf:Jacobi-estimate}
That the constants of the estimates in the lemma are \emph{independent} of $\delta$
is due to the uniform estimates (in $\delta$) of the nonlinear term and to the
fact that the Malliavin derivatives $\mathcal{D}_v \Phi_t$ do not depend on $v^H$.

The proofs of \eqref{e:Jacobi-estimate1}, \eqref{e:Jacobi-estimate2} are
classical since the SDEs for $J_t$, $J^{-1}_t$ are both finite dimensional and have the cutoff.
The proof of \eqref{e:solution-estimate2} is by the same procedure as for \eqref{e:DHL}.
For the other estimates, we will apply the bootstrap argument in the proof
of \eqref{e:solution-estimate1} but omit the trivial induction argument.

As for \eqref{e:J-1}, we consider the integral form of equation \eqref{e:J-1t}
and obtain by applying some classical inequalities
\begin{equation*}
\begin{split}
3^{-p}|J^{-1}_{t}h^L|^p_\W &\leq
|h^L|^p_\W+t^{p/q} \int_0^t
|J^{-1}_{s}[A_L+D_L(B_L(\Phi_s, \Phi_s)\chi
(\frac{|\Phi|_\W}{3\rho}))
-\tr((D_LQ_L(\Phi_t))^2)]h^L|^p_\W dt \\
&\quad +\Bigl|\int_0^t J^{-1}_{s} D_LQ_L(\Phi_s)h^LdW^L_s\Bigr|^p_\W.
\end{split}
\end{equation*}
Since all the operators in the above inequalities are finite
dimensional, by \eqref{e:B-uv1}, Doob's martingale
inequality and Birkhold-Davis-Gundy inequality, one has
\begin{equation*}
\E\bigl[\sup_{0 \leq t \leq T}|J^{-1}_{t}h^L|^p_\W\bigr]
\leq C_1 \Bigl(1+T^p \E\bigl[\sup_{0 \leq t \leq T}
|J^{-1}_{t}|^p_{\mathcal{L}(\W)}\bigr]+T^{\frac{p}2}\E\bigl[\sup_{0 \leq t \leq T}
|J^{-1}_{t}|^p_{\mathcal{L}(\W)}\bigr]\Bigr)|h^L|^p_\W
\end{equation*}
where $C_1=C_1(p, \rho, \alpha_0)$. When $T$ is small enough, we have
$\E[\sup_{0 \leq t \leq
T}|J^{-1}_{t}|^p_{\mathcal{L}(\W)}] \leq
\frac{C_1}{1-C_1(T^{p}+T^{p/2})}$.

Clearly, $\mathcal{D}_v \Phi^L_t$ satisfies the following equation
\begin{align*}
\mathcal{D}_v \Phi^L_t
& =\int_0^t\mathrm{e}^{-A(t-s)}[-B_L(\Phi_s,\mathcal{D}_v \Phi^L_s)-B_L(\mathcal{D}_v
\Phi^L_s,\Phi_s)] \chi (\frac{|\Phi_s|_\W}{3\rho})\,ds\\
&\quad -\frac{1}{3\rho}\int_0^t \mathrm{e}^{-A(t-s)}B_L(\Phi_s,\Phi_s)
\chi'(\frac{|\Phi_s|_\W}{3\rho})\frac{\langle D_{v}
\Phi^L_s,\Phi_s\rangle_\W}{|\Phi_s|_\W}\,ds \\
&\quad +\int_0^t\! \mathrm{e}^{-A(t-s)} (1-
\chi'(\frac{|\Phi_s|_\W}{\rho}))Q_Lv^L\,ds
 -\frac{1}{\rho}\int_0^t \mathrm{e}^{-A(t-s)}\chi'(\frac{|\Phi_s|_\W}{\rho})\frac{\langle D_{v}
\Phi^L_s,\Phi_s\rangle_\W}{|\Phi_s|_\W}Q_L\,dW^L_s \\
&=J_1(t)+J_2(t)+J_3(t)+J_4(t)
\end{align*}
By \eqref{e:B-uv1} and Lemma \ref{lem:noise-estimate}, one has
\begin{align*}
&|J_1(t)|_\W \leq \int_0^t \frac{C_2}{\sqrt{t-s}}
|\mathcal{D}_v \Phi^L_s|_\W\,ds \\
&|J_2(t)|_\W  \leq \int_0^t \frac{C_3}{\sqrt{t-s}}
|\mathcal{D}_v \Phi^L_s|_\W\,ds \\
&\E\Bigl(\sup_{0 \leq t \leq T}|J_3(t)|^p_\W\Bigr) \leq
C_4
\E \Bigl( \int_0^T |v^L(s)|^p_\W\,ds\Bigr)  \\
&\E\Bigl(\sup_{0 \leq t \leq T}|J_4(t)|^p_\W\Bigr) \leq
C_5 T^{p/8} \E \Bigl(\sup_{0 \leq t \leq T}
|\mathcal{D}_{v} \Phi^L_t|^p_\W \Bigr),\qquad 0\leq T \leq 1,
\end{align*}
with $C_i=C_i(\rho, \alpha_0)$ ($i=2,3$) and $C_i=C_i(\rho, \alpha_0, p)$ ($i=4,5$). Thus, for $p \geq 2$,
\begin{equation*}
\E\Bigl(\sup_{0 \leq t \leq T}|\mathcal{D}_{v}
\Phi^L_t|^p_\W \Bigr) \leq C_6 T^{p/8}
\E\Bigl (\sup_{0 \leq t \leq T}|\mathcal{D}_{v}
\Phi^L_t|^p_\W \Bigr)+C_6 \E\bigl(\int_0^T
|v^L(s)|^p_\W\,ds\bigr)
\end{equation*}
with $C_6=C_6(\alpha_0,\rho, p)$, and
$\E\bigl(\sup_{0 \leq t \leq T}|\mathcal{D}_{v}
\Phi^L_t|^p_\W \bigr) \leq \tfrac{C_6}{1-C_6 T^{p/8}} \E\bigl(\int_0^T
|v^L(s)|^p_\W\,ds\bigr)$  for $T$ small enough.

The term $\mathcal{D}_{v_1} \mathcal{D}_{v_2} \Phi_t$
satisfies the following equation
\begin{align*}
\mathcal{D}_{v_1} \mathcal{D}_{v_2} \Phi^L_t&=-\int_0^t
\mathrm{e}^{-A(t-s)}\mathcal{D}_{v_1} \mathcal{D}_{v_2}(B_L(\Phi_s, \Phi^L_s)
\chi (\frac{|\Phi_s|_\W}{3\rho}))\,ds\\
&\quad +\int_0^t \mathrm{e}^{-A(t-s)} \mathcal{D}_{v_2}Q_L(\Phi_s)v^L_1(s)\,ds
+\int_0^t \mathrm{e}^{-A(t-s)} \mathcal{D}_{v_1}
\mathcal{D}_{v_2}Q_L(\Phi_s)dW^L_s
\end{align*}
Expanding the terms $\mathcal{D}_{v_1} \mathcal{D}_{v_2}(B_L(\Phi_s,
\Phi^L_s) \chi (\frac{|\Phi_s|_\W}{3\rho}))$ and
$\mathcal{D}_{v_1} \mathcal{D}_{v_2}Q_L(\Phi_s)$, we obtain two very
complex expressions which we omit them but point out the key points
for their estimates. Noticing the fact $\mathcal{D}_{v_2}
\Phi_t=\mathcal{D}_{v_2} \Phi^L_t$, $|\Phi_t|_\W \chi (\frac{|\Phi_t|_\W}{3\rho}) \leq 3
\rho$, and using \eqref{e:B-uv1} and  Lemma
\ref{lem:noise-estimate}, one has
\begin{gather*}
|\mathrm{e}^{-A(t-s)} \mathcal{D}_{v_2}Q_L(\Phi_s) v^L_1(s)|_\W
\leq C_7|\mathcal{D}_{v_2} \Phi^L_t|_\W
|v^L_1|_\W,\\
\Bigl|\mathrm{e}^{-A(t-s)}\mathcal{D}_{v_1} \mathcal{D}_{v_2}(B_L(\Phi_s, \Phi_s)
\chi (\frac{|\Phi_s|_\W}{3\rho}))\Bigr|_\W \leq
\frac{C_8}{\sqrt{t-s}} \bigl(|\mathcal{D}_{v_1} \mathcal{D}_{v_2}
\Phi^L_t|_\W+ |\mathcal{D}_{v_1}
\Phi^L_t|_\W|\mathcal{D}_{v_2} \Phi^L_t|_\W\bigr),
\end{gather*}
and
\begin{multline*}
\E\Bigl(\sup_{0 \leq t \leq T}|\int_0^t \mathrm{e}^{-A(t-s)}
\mathcal{D}_{v_1}\mathcal{D}_{v_2}Q_L(\Phi_s)
dW^L_s|^p_\W\Bigr)\leq\\
\leq C_9 T^{p/8} \E\Bigl[\sup_{0 \leq t \leq T}(|\mathcal{D}_{v_1}
\mathcal{D}_{v_2}\Phi^L_t|^p_\W + |\mathcal{D}_{v_1}\Phi^L_t|^p_\W
|\mathcal{D}_{v_2}\Phi^L_t|^p_\W)\Bigr],
\end{multline*}
for $0<T\leq 1$, with $C_i=C_i(\rho,\alpha_0)$ ($i=7,8$) and $C_9=C_9(\rho, \alpha_0, p)$.
Hence, when $T$ is small
\begin{multline*}
\E\Bigl(\sup_{0 \leq t \leq T}|\mathcal{D}_{v_1}
\mathcal{D}_{v_2}\Phi^L_t|^p_\W\Bigr) \leq \frac{C_9}{1 - C_9 T^{p/8}}
\E\Bigl(|\mathcal{D}_{v_1}\Phi^L_t|^p_\W|\mathcal{D}_{v_2}\Phi^L_t|^p_\W\Bigr)\leq \\
\leq \Bigl(\frac{C_{10}}{1-C_{10} T^{p/8}}\Bigr)^2
\Bigl(1+\E[\int_0^T |v^L_1(s)|_\W^{2p}\,ds] \Bigr)^{\frac12}
\Bigl(1+\E[\int_0^T |v^L_2(s)|_\W^{2p}\,ds] \Bigr)^{\frac12},
\end{multline*}
with $C_{10}=C_{10}(\rho, \alpha_0, p)$. The proof of \eqref{e:mix-two-order-estimate} is similar.
\end{proof}
\bibliographystyle{amsplain}

\begin{thebibliography}{99}

\bibitem{AX09}
	S. Albeverio, L. Xu,
	\emph{Ergodicity of 3D stochastic Navier-Stokes equation driven by mildly degenerate noises: Kolmogorov equation approach} (in preparation).

\bibitem{BloFlaRom09}
	D. Bl{\"o}mker, F. Flandoli, M. Romito,
	\emph{Markovianity and ergodicity for a surface growth PDE},
	Ann. Probab. 37, no. 1 (2009), 275-313.

\bibitem{DPD}
	G. Da Prato, A. Debussche,
	\emph{Ergodicity for the 3D stochastic Navier-Stokes equations},
	J. Math. Pures Appl. (9), 82 (2003), 877--947.

\bibitem{DPZ}
	G. Da Prato, J. Zabczyk,
	\emph{Stochastic equations in infinite dimensions},
	vol. 44 of Encyclopedia of Mathematics and its Applications,
	Cambridge University Press, Cambridge, 1992.

\bibitem{DPZ1}
	G. Da Prato, J. Zabczyk,
	\emph{Ergodicity for infinite-dimensional systems},
	vol. 229 of London Mathematical Society Lecture Note Series,
	Cambridge University Press, Cambridge, 1996.

\bibitem{DO} A. Debussche, C. Odasso,
	\emph{Markov solutions for the 3D stochastic Navier-Stokes equations with state dependent noise},
	J. Evol. Equ., 6 (2006), 305--324.

\bibitem{EM}
	W. E, J. C. Mattingly,
	\emph{Ergodicity for the Navier-Stokes equation with degenerate random forcing: finite-dimensional approximation},
	Comm. Pure Appl. Math., 54 (2001), 1386--1402.

\bibitem{EH}
	J.-P. Eckmann, M. Hairer,
	\emph{Uniqueness of the invariant measure for a stochastic PDE driven by degenerate noise},
	Comm. Math. Phys., 219 (2001), 523--565.

\bibitem{Fef06}
	C. L. Fefferman,
	\emph{Existence and smoothness of the Navier-Stokes equation},
	in The millennium prize problems, Clay Math. Inst., Cambridge, MA, 2006, 57--67.

\bibitem{FG}
	F. Flandoli, D. G{\polhk{a}}tarek,
	\emph{Martingale and stationary solutions for stochastic Navier-Stokes equations},
	Probab. Theory Related Fields, 102 (1995), 367--391.

\bibitem{FM}
	F. Flandoli, B. Maslowski,
	\emph{Ergodicity of the $2$-D Navier-Stokes equation under random perturbations},
	Comm. Math. Phys., 172 (1995),  119--141.

\bibitem{Flarom06}
	F. Flandoli, M. Romito,
	\emph{Markov selections and their regularity for the three-dimensional stochastic Navier-Stokes equations},
	C. R. Math. Acad. Sci. Paris, 343 (2006),  47--50.

\bibitem{FlaRom07}
	F. Flandoli, M. Romito,
	\emph{Regularity of transition semigroups associated to a 3D stochastic Navier-Stokes equation},
	in Stochastic differential equations: theory and applications,
	vol. 2 of Interdiscip. Math. Sci., World Sci. Publ., Hackensack, NJ, 2007, 263--280.

\bibitem{FR}
	F. Flandoli, M. Romito,
	\emph{Markov selections for the three-dimensional stochastic Navier-Stokes equations},
	Probab. Theory Relat. Fields, 140 (2008), 407--458.

\bibitem{GolRocZha08}
	B. Goldys, M. R{\"o}ckner, X. Zhang,
	\emph{Martingale solutions and markov selections for stochastic partial differential equations},
	2008. Preprint.

\bibitem{HM1}
	M. Hairer, J. C. Mattingly,
	\emph{Ergodicity of the 2D Navier-Stokes equations with degenerate stochastic forcing},
	Ann. of Math. (2), 164 (2006),  993--1032.

\bibitem{HM2}
	M. Hairer, J. C. Mattingly,
	\emph{A theory of hypoellipticity and unique ergodicity for semilinear stochastic PDEs},
	2008, arXiv:0808.1361 [math.PR].

\bibitem{L}
	 J. Leray,
	\emph{Sur le mouvement d'un liquide visqueux emplissant l'espace},
	Acta Math., 63 (1934),  193--248.

\bibitem{No}
	J. Norris,
	\emph{Simplified {M}alliavin calculus},
	in S\'eminaire de Probabilit\'es, XX, 1984/85,
	vol. 1204 of Lecture Notes in Math., Springer, Berlin, 1986,  101--130.

\bibitem{RZ}
	M. R\"ockner, X. Zhang,
	\emph{Stochastic tamed 3d Navier-Stokes equations: existence, uniqueness and ergodicity},
	2008, arXiv:0802.3934 [math.PR].

\bibitem{R1}
	M. Romito,
	\emph{Ergodicity of the finite dimensional approximation of the 3D Navier-Stokes equations forced by a degenerate noise},
	J. Statist. Phys., 114 (2004),  155--177.

\bibitem{R}
	M. Romito,
	\emph{Analysis of equilibrium states of Markov solutions to the 3D Navier-Stokes equations driven by additive noise},
	J. Stat. Phys., 131 (2008), 415--444.

\bibitem{Rom08a}
	M. Romito,
	\emph{The martingale problem for Markov solutions to the navier-stokes equations},
	2008. Submitted for the proceedings of the $6^\text{th}$ Ascona conference
	\emph{Stochastic analysis, random fields and applications VI}.

\bibitem{Rom08b}
	M. Romito,
	\emph{An almost sure energy inequality for Markov solutions to the 3D Navier-Stokes equations},
	2008. Submitted for the proceedings of the conference
	\emph{Stochastic partial differential equations and applications VIII}.

\bibitem{Shi06}
	A. Shirikyan,
	\emph{Approximate controllability of three-dimensional Navier-Stokes equations},
	Comm. Math. Phys., 266 (2006),  123--151.

\bibitem{T}
	R. Temam,
	\emph{Navier-Stokes equations and nonlinear functional analysis},
	vol. 66 of CBMS-NSF Regional Conference Series in Applied Mathematics,
	Society for Industrial and Applied Mathematics (SIAM), Philadelphia, PA, second ed., 1995.

\end{thebibliography}

\end{document}